\documentclass[reqno]{amsart}
\usepackage{amsmath, amssymb, amsthm, geometry, enumerate, graphicx, amsfonts, hyperref, mathrsfs}
\hypersetup{colorlinks=true,linkcolor=red, anchorcolor=blue, citecolor=blue, urlcolor=red, filecolor=magenta, pdftoolbar=true}
\theoremstyle{plain}
\newtheorem{thm}{\it Theorem}[section]
\newtheorem{prop}[thm]{\it Proposition}

\theoremstyle{remark}
\newtheorem{defn}[thm]{Definition}
\newtheorem{rem}[thm]{Remark}
\newtheorem{ex}[thm]{Example}
\numberwithin{equation}{section}
\usepackage{enumerate}
\allowdisplaybreaks
\begin{document}

\title  [Matrix-Valued Gabor Frames over LCA Groups for Operators]{Matrix-Valued Gabor Frames over LCA Groups for Operators}

\author[Jyoti]{Jyoti}

\address{Department of Mathematics,
University of Delhi, Delhi-110007.}
\email{jyoti.sheoran3@gmail.com}

\author[Lalit Kumar  Vashisht]{Lalit  Kumar Vashisht$^*$}
\address{
Department of Mathematics,
University of Delhi, Delhi-110007.}
\email{lalitkvashisht@gmail.com}
\author[Uttam Kumar Sinha]{Uttam Kumar Sinha}
\address{
Department of Mathematics,  Shivaji College,
University of Delhi, Delhi-110027, India.}
\email{uksinha@shivaji.du.ac.in}

\begin{abstract}
G\v avruta studied atomic systems   in terms of frames for range of operators (that is, for subspaces),  namely $K$-frames, where the lower frame condition is controlled by the Hilbert-adjoint of a bounded linear operator $K$. For  a  locally compact abelian group G and a positive integer $n$, we study frames of matrix-valued Gabor systems in the matrix-valued Lebesgue space $L^2(G, \mathbb{C}^{n\times n})$ , where a bounded linear operator $\Theta$ on $L^2(G, \mathbb{C}^{n\times n})$ controls not only lower but also the upper frame condition. We term such frames matrix-valued $(\Theta, \Theta^*)$-Gabor frames. Firstly, we discuss frame preserving mapping in terms of hyponormal operators. Secondly, we give necessary and sufficient conditions for the existence of matrix-valued $(\Theta, \Theta^*)$- Gabor frames in terms of hyponormal operators. It is shown that if $\Theta$ is adjointable hyponormal operator, then $L^2(G, \mathbb{C}^{n\times n})$ admits a $\lambda$-tight $(\Theta, \Theta^*)$-Gabor frame for every positive real number $\lambda$. A characterization of matrix-valued $(\Theta, \Theta^*)$-Gabor frames is given. Finally, we show that matrix-valued $(\Theta, \Theta^*)$-Gabor frames are stable under small perturbation of window functions. Several examples are given to support our study.
\end{abstract}

\renewcommand{\thefootnote}{}

\footnote{
The research of  Lalit Kumar Vashisht is  supported by the Faculty Research Programme Grant-IoE, University of Delhi \ (Grant No.: Ref. No./IoE/2021/12/FRP).\\
$^*$Corresponding author: Lalit Kumar Vashisht.}

\maketitle

\baselineskip15pt
 \section{Introduction}
In a fundamental published paper  \cite{G},  Gabor introduced a fundamental approach to signal decomposition in terms of elementary
signals.  Duffin and Schaeffer \cite{DS} in
1952, while addressing some deep problems in non-harmonic Fourier
series, abstracted Gabor's method to define frames for Hilbert
spaces. To be exact, they introduced frames of exponentials for the space $L^2(-\delta, \delta)$ under the name Fourier frame. Let $\mathcal{H}$ be a  complex separable   Hilbert space with an inner product $\langle ., . \rangle$. A countable collection of vector $\Phi:=\{\varphi_k\}_{k=1}^{\infty}$ in a separable Hilbert space $\mathcal{H}$   is called a \emph{frame} (or \emph{Hilbert frame})  for  $\mathcal{H}$  if there exist finite positive scalars $A_o$ $B_o$  such that
\begin{align}\label{fin1.1}
A_o \|\varphi\|^2\leq  \sum_{k = 1}^{\infty} |\langle \varphi, \varphi_k\rangle|^2 \leq B_o \|\varphi\|^2
\end{align}
for all $\varphi \in \mathcal{H}$.
The scalars $A_o$ and $B_o$ are called the \emph{lower} and \emph{upper frame bounds} of $\Phi$, respectively. Ineq. \eqref{fin1.1} is called the  \emph{frame inequality} of $\Phi$. The frame inequality guarantee invertibility of the frame operator  $\mathcal{S} : \mathcal{H}\rightarrow \mathcal{H}$ given by $\mathcal{S}\varphi= \sum\limits_{k = 1}^{\infty} \langle \varphi, \varphi_k\rangle \varphi_k$. This gives the stable reconstruction of each $\varphi$ in $\mathcal{H}:$
$\varphi = \mathcal{S}\mathcal{S}^{-1}\varphi =\sum\limits_{k = 1}^{\infty} \langle \mathcal{S}^{-1}\varphi, \varphi_k \rangle \varphi_k = \sum\limits_{k = 1}^{\infty} \langle \varphi, \mathcal{S}^{-1}\varphi_k \rangle \varphi_k.$ This decomposition is useful in signal processing \cite{Kova07, Vst}, in particular, in lost of coefficients, see \cite{OC2, H11} for technical details. Nowadays, frames are used in sampling \cite{Asamp}, iterated function system \cite{Vlk}, distributed signal processing \cite{DHou}, operator theory \cite{Ald2, DVV17, JVashII, JSchmidt}, application of wavelets \cite{H89}.   We refer to texts \cite{OC2, Groch, H11, Y} for basic theory of frames.

G$\check{a}$vruta in \cite{LG} introduced the notion of $K$-frames, where $K$ is a linear bounded  operator acting on the underlying Hilbert  space $\mathcal{H}$.
\begin{defn}\cite[ p.  142]{LG}
Let $K \in \mathcal{B}(\mathcal{H})$.  A sequence $\Phi:=\{\varphi_k\}_{k\in\mathbb{I}} \subset
\mathcal{H}$ is called a \emph{$K$-frame} for $\mathcal{H}$ if there exist
constants $0 < a_o, \  b_o < \infty$ such that
\begin{align}
a_o\|K^*\varphi\|^2\leq \sum_{k =1}^{\infty} |\langle \varphi_k, \varphi \rangle|^2 \leq
b_o\|\varphi\|^2 \ \text{for all} \ \varphi \in \mathcal{H}.
\end{align}
\end{defn}
The numbers $a_o$ and $b_o$ are collectively known as \emph{$K$-frame bounds}. If $K=I$,  the identity operator on $\mathcal{H}$, then $K$-frames are the ordinary Hilbert frames. However, a $K$-frame need not be a frame when $K\neq I$. To be exact, $K$-frames are generalization of frames, which allow the reconstruction of elements from the range Ran$(K)$ of $K$. Note that a $K$-frame  for $\mathcal{H}$ is a Bessel sequence, so its frame operator  is well defined. But,  in general, it is  not invertible on $\mathcal{H}$. However, the frame operator of a $K$-frame  is invertible on the subspace Ran$(K)$ of $\mathcal{H}$, whenever the Ran$(K)$  is closed. In  \cite{LG},  G$\check{a}$vruta  characterized  $K$-frames in separable Hilbert spaces by using bounded linear operators on the underlying space. $K$-frames are also related to atomic systems and G$\check{a}$vruta in \cite{LG} characterized atomic systems in terms of  $K$-frames in separable  Hilbert spaces. She also observed  many differences between $K$-frames and ordinary  frames in separable Hilbert spaces. More precisely, $K$-frames gives stable analysis and reconstruction of functions from a subspace, e.g., range of operators.  Xiao, Zhu, and  G$\check{a}$vruta  \cite{XZG}  gave various methods to construct $K$-frames in separable Hilbert spaces. They also discussed stability  of $K$-frames under small perturbation. Recently, $K$-frames in distributed signal processing are studied in \cite{DV17ar, JVash, XChen}.

Frames  in matrix-valued signal spaces  have potential applications in signal processing  as most of the application  areas involve  matrix-valued  signals.  Xia and  Suter in \cite{XS} studied vector-valued wavelets which play important role in multivariate signals.  It is worth observing that frame properties, in general, not  carried from  a signal space  to  its associated matrix-valued signal space.   In this direction, the authors of \cite{JyVa20} studied an interplay between frames and matrix-valued frames, where they considered the wave packet structure in the euclidean matrix-valued space $L^2(\mathbb{R}^d, \mathbb{C}^{s\times r})$. They also gave some classes of matrix-valued window functions which can generate frames. Frame properties of WH-packets which is generalized Aldroubi's model \cite{A} for construction of new frames from a given frame studied in \cite{JVIs} and sufficient conditions for finite sums of matrix-valued wave packet frames can be found in \cite{JDV2}. Two authors in \cite{JVash} introduced and studied matrix-valued frames for range of operators. Recently, matrix-valued Gabor frames over locally compact abelian groups studied by authors of \cite{Jind}. Motivated by applications of matrix-valued frames and differences between ordinary frames and $K$-frames, we study matrix-valued Gabor frames over locally compact abelian (LCA) groups, where both the lower frame condition and upper frame condition are controlled by bounded linear operators, in particular hyponormal operators, on the matrix-valued signal space over LCA groups. Notable contribution in this work include frame preserving mapping in terms of hyponormal operators, existence of tight matrix-valued Gabor frames over LCA groups for hyponormal operators. A characterization of matrix-valued Gabor frames over LCA groups and new stability results for  matrix-valued Gabor frames over LCA groups under small perturbation.

This paper is organized as follows. In Section \ref{SecII}, we set the basic notions and definitions
on the  matrix-valued signal space and matrix-valued Gabor frames over locally compact abelian (LCA) groups and frames for operators to the make the paper self-contained. We introduce
matrix-valued $(\Theta, \Theta^*)$-Gabor frames in the matrix-valued signal space $ L^2(G, \mathbb{C}^{n\times n})$ over LCA groups in Section \ref{SecIII}, where $G$ is a LCA group, $n$ is a positive integer and $\Theta$ is a bounded linear operator acting on  $ L^2(G, \mathbb{C}^{n\times n})$. In $(\Theta, \Theta^*)$-Gabor frames both the lower frame condition and upper frame condition are controlled by $\Theta$.
 Proposition \ref{prop3.4} gives sufficient condition for a matrix-valued Gabor frame to be $(\Theta, \Theta^*)$-Gabor frame in terms of bounded belowness of $\Theta$. Frame preserving maps in terms of hyponormal operators are given in Proposition \ref{prop1} and Proposition \ref{1a}. Theorem \ref{thm2} provides existence of tight matrix-valued $(\Theta, \Theta^*)$-Gabor frames in  $ L^2(G, \mathbb{C}^{n\times n})$. Proposition \ref{2a} shows that $(\Theta, \Theta^*)$-Gabor frames are preserved under adjointable hyponormal  operators. A characterization for the existence of  $(\Theta, \Theta^*)$-Gabor frames in $ L^2(G, \mathbb{C}^{n\times n})$ is given in Theorem \ref{3.5}. Two different perturbation results which gives stability of frame conditions  in terms of window functions and operators are given in Theorem \ref{pert} and Theorem \ref{sum}. Examples and counter-examples are given to illustrate our results.

\section{Preliminaries}\label{SecII}
Throughout the paper, symbol $\mathbb{Z}$ and $\mathbb{C}$ denote the set of integers and complex numbers, respectively. $\mathbb{T}$ denote the unit circle group. Let $G$ be a second countable locally compact abelian group equipped with the  Hausdorff topology. We recall that  a character on $G$ is the map  $\gamma$  $G$ into itself which satisfies $\gamma(x+y)=\gamma(x)\gamma(y)$ for all $x,y\in G$. The \emph{dual group} of $G$, denoted by $\hat{G}$, is the collection of all continuous characters on $G$ which forms a locally compact abelian group under the operation defined by $(\gamma + \gamma')(x):= \gamma(x)\gamma'(x)$, where $ \gamma,\gamma'\in \hat{G}$ and $x\in G$ and an appropriate topology.  It is well known that on a LCA group $G$ there exists a Haar measure which is unique upto a positive scalar multiple, see \cite{Foll} for details.
The symbols $\mu_{G}$ and $\mu_{\hat{G}}$ denote the Haar measure on $G$ and $\hat{G}$, respectively. A \emph{lattice} of $G$ is a discrete subgroup $\Lambda$ of $G$ for which $G/\Lambda$ is compact. The annihilator of $\Lambda$, denoted by $\Lambda^\perp$,  is defined by $\Lambda^\perp= \{\gamma \in \hat{G}\ | \ \gamma(x)=1,  \ x\in \Lambda\}$. Note that  $\Lambda^\perp$  is a lattice in  $\hat{G}$. The \emph{fundamental domain} associated with the lattice $\Lambda^\perp$ of $\hat{G}$, denoted by $V$, is a Borel measurable relatively compact set in $\hat{G}$ such that $\hat{G}= \cup_{w\in \Lambda^{\perp}}(w+V), \ (w+V)\cap(w'+V) = \emptyset$ for $w\neq w', w,w'\in \Lambda^\perp$. The collection of all continuous automorphisms on $G$ is denoted by  Aut$G$. As is standard $L^2(G)$ denote the  space of measurable square integrable functions over $G$.  The   \emph{Fourier transform} of a function $f$ in  $L^{1}\bigcap L^{2}(G)$ is defined as
\begin{align*}
\widehat{f}(\gamma) =  \int_{G} f(x) \overline{\gamma(x)} d\mu_{G}(x), \ \ \gamma \in \widehat{G}
\end{align*}
Note that the Fourier transform can be extended isometrically to $L^{2}(G)$, see \cite{Foll}.

  \subsection{The Space $ L^2(G, \mathbb{C}^{n\times n})$} Throughout the paper, the  matrix-valued functions are denoted by bold letters. Let $n$ be a positive integer. The space of matrix-valued functions over $G$, denoted by $ L^2(G, \mathbb{C}^{n\times n})$, is defined as
\begin{align*}
 L^2(G, \mathbb{C}^{n\times n}) := \Big\{\mathbf{f} = \big[f_{i j}\big]_{1 \leq i, j \leq n}
 :  f_{ij} \in  L^2(G)\ (1 \leq i, \ j \leq n)\Big\},
\end{align*}
where $\big[f_{i j}\big]_{1 \leq i, j \leq n}$ is matrix of order $n$ with entries $f_{i j}$. The functions $f_{ij}$  are called \emph{components} or \emph{atoms}  of $\mathbf{f}$. The  Frobenius norm on $L^2(G, \mathbb{C}^{n\times n})$ is given by
 \begin{align}\label{eqnorm1}
\|\mathbf{f}\| = \Big(\sum\limits_{i,j =1}^{n}\int_{G}|f_{ij}|^2 d\mu_{G} \Big)^{\frac{1}{2}}.
\end{align}
It is easy to see that  $L^2(G, \mathbb{C}^{n\times n})$ is a Banach space with respect to the Frobenius norm given in \eqref{eqnorm1}.

The integral of   a function
 $\mathbf{f} = \big[f_{i j}\big]_{1 \leq i, j \leq n} \in L^2(G, \mathbb{C}^{n\times n})$  is defined as
\begin{align*}
\int_{G}\mathbf{f}d\mu_{G} = \left[\int_{G}f_{i j}d\mu_{G}\right]_{1 \leq i, j \leq n}
\end{align*}
For $\mathbf{f},  \mathbf{g} \in L^2(G, \mathbb{C}^{n\times n})$, the matrix-valued inner product is defined as
 \begin{align}\label{dminp}
\langle \mathbf{f},  \mathbf{g}\rangle = \int_{G}\mathbf{f}(x) \mathbf{g}^*(x) d\mu_{G}.
\end{align}
Here, $*$ denotes the transpose and the complex conjugate. One may observe that the matrix-valued inner product given in  \eqref{dminp}  is not an  inner product in usual sense. Further, a bounded linear operator on $ L^2(G, \mathbb{C}^{n\times n})$ may not be adjointable with respect to the matrix-valued product given in \eqref{dminp}.

Let  tr$A$ denotes trace of the matrix $A$. The space $L^2(G, \mathbb{C}^{n \times n})$ becomes a Hilbert space  with respect to the inner product  $\langle \cdot,\cdot \rangle_{o}$ defined by
 \begin{align*}
  \langle \mathbf{f},\mathbf{g}\rangle_o = \text{tr}\langle \mathbf{f},\mathbf{g}\rangle, \  \ \mathbf{f},  \  \mathbf{g} \in L^2(G, \mathbb{C}^{n\times n}),
   \end{align*}
   and $\langle \cdot,\cdot \rangle_{o}$ generates the Frobenius norm:
 $||\mathbf{f}||^2  = \langle \mathbf{f},\mathbf{f} \rangle_o, \ \mathbf{f} \in L^2(G, \mathbb{C}^{n\times n})$.

\begin{defn}
A bounded linear operator $U$ on $ L^2(G, \mathbb{C}^{n\times n})$ is said to be \emph{hyponormal} if $\text{tr}\langle UU^*\mathbf{f}, \mathbf{f}\rangle \leq \text{tr}\langle U^*U\mathbf{f}, \mathbf{f}\rangle$, for all
$ \mathbf{f} \in L^2(G, \mathbb{C}^{n\times n})$. That is,  $\|U^* \mathbf{f}\| \leq \|U\mathbf{f}\|$ for all  $\mathbf{f} \in L^2(G, \mathbb{C}^{n\times n})$.
\end{defn}
For fundamental properties of hyponormal operators, we refer to \cite{Hypo}.
\subsection{Matrix-Valued Gabor Frames in $L^2(G, \mathbb{C}^{n\times n})$}
 Let $\Lambda_0$ be  a  finite subset of $\mathbb{N}$,  $B\in$ Aut$G$, $C\in$ Aut$\widehat{G}$, $\Lambda$ be a lattice in $G$ and  $\Lambda'$  a lattice in $\widehat{G}$.

Write
\begin{align*}
& \Phi_{\Lambda_0}: = \{\Phi_l\}_{l \in \Lambda_0} \subset  L^2(G, \mathbb{C}^{n\times n}),\\
& \texttt{G}(C, B, \Phi_{\Lambda_0}) : =  \{E_{Cm}T_{Bk}\Phi_l\}_{l\in \Lambda_0,k\in \Lambda,m\in \Lambda'} \subset L^2(G, \mathbb{C}^{n\times n}).
\end{align*}
For $a\in G$ and $\eta \in \hat{G}$, we consider following operators  on  $ L^2(G, \mathbb{C}^{n\times n})$.
\begin{align*}
   T_a\mathbf{f}(x) &=\mathbf{f}(xa^{-1}) \quad (\text{Translation operator} ),\\
 E_\eta\mathbf{f}(x) & =\eta(x)\mathbf{f}(x) \quad (\text{Modulation operator}).
\end{align*}
For $l\in \Lambda_0$, let $\Phi_l \in L^2(G, \mathbb{C}^{n\times n})$ be given by $\Phi_l(x) = \begin{bmatrix} \phi_{ij}^{(l)}(x) \end{bmatrix}_{n\times n}$. Let \break $B \in $  Aut$G$   and $C \in$ Aut$\widehat{G}$.  A collection  of the form
\begin{align*}
\texttt{G}(C, B, \Phi_{\Lambda_0}) : =  \{E_{Cm}T_{Bk}\Phi_l\}_{l\in \Lambda_0,k\in \Lambda,m\in \Lambda'}
\end{align*}
is called the  \emph{matrix-valued Gabor system} in the  space  $L^2(G, \mathbb{C}^{n\times n})$ over LCA group $G$.  The functions $\Phi_l$ are called the \emph{matrix-valued Gabor  window  functions}.

\begin{defn}
 A frame of the form $\texttt{G}(C, B, \Phi_{\Lambda_0})$   for $L^2(G, \mathbb{C}^{n\times n})$ is called a \emph{matrix-valued Gabor frame}. That is, the inequality (\emph{frame inequality})
\begin{align*}
a_o \|\mathbf{f}\|^2\leq \sum _{l\in \Lambda_0}\sum_{k\in \Lambda,m\in \Lambda'}\Big\|\Big\langle  E_{Cm}T_{Bk}\Phi_l,\mathbf{f} \Big\rangle \Big\|^2 \leq b_{o} \|\mathbf{f}\|^2, \ \mathbf{f} \in L^2(G, \mathbb{C}^{n\times n}),
\end{align*}
holds for  some positive scalars $ a_{o}$ and $b_{o}$. As in case of ordinary frames, $ a_{o}$ and $b_{o}$ are called frame bounds.
\end{defn}
 Let $\mathcal{M}_n(\mathbb{C})$ be the  complex vector space of all $n \times n$ complex matrices. The space
\begin{align*}
&\ell^2(\Lambda_0 \times \Lambda \times \Lambda',\mathcal{M}_n(\mathbb{C}))\\
& := \Big\{\{M_{l,j,k}\}_{l \in \Lambda_0,j \in \Lambda, k \in \Lambda'}\subset  \mathcal{M}_n(\mathbb{C}): \sum_{l \in \Lambda_0}\sum\limits_{j \in \Lambda, k \in \Lambda'}\|M_{l,j,k}\|^2<\infty\Big\}
 \end{align*}
 is a Hilbert space  and its related  norm is given by
\begin{align*}
 \| \{M_{l,j,k}\}_{l \in \Lambda_0, j \in \Lambda, k \in \Lambda'} \| = \Big( \sum\limits_{l \in \Lambda_0} \sum\limits_{ j \in \Lambda, k \in \Lambda'}\|M_{l,j,k}\|^2\Big)^{\frac{1}{2}}.
\end{align*}

If $\texttt{G}(C, B, \Phi_{\Lambda_0})$ is a frame   for $L^2(G, \mathbb{C}^{n\times n})$, then the map
\begin{align*}
&V: \ell^2(\Lambda_0 \times \Lambda \times \Lambda',\mathcal{M}_n(\mathbb{C})) \rightarrow L^2(G, \mathbb{C}^{n\times n}) \  \text{defined by}\\
&V:  \{ M_{l,k,m} \}_{l\in \Lambda_0,k\in \Lambda,m\in \Lambda'}  \mapsto \sum _{l\in \Lambda_0} \sum_{k\in \Lambda,m\in \Lambda'} M_{l,k,m} E_{Cm}T_{Bk}\Phi_l
\end{align*}
is called the   \emph{synthesis operator} (or the \emph{pre-frame operator}), associated with $\texttt{G}(C, B, \Phi_{\Lambda_0})$.
The \emph{analysis}  \emph{operator} is the  map
\begin{align*}
& W:L^2(G, \mathbb{C}^{n\times n}) \rightarrow     \ell^2(\Lambda_0\times \Lambda \times \Lambda',\mathcal{M}_n(\mathbb{C})) \ \text{given  by}\\
&W:  \mathbf{f}  \mapsto  \Big\{ \langle \mathbf{f},E_{Cm}T_{Bk}\Phi_l \rangle\Big\}_{l\in \Lambda_0,k\in \Lambda,m\in \Lambda'}.
\end{align*}
The frame operator of $\texttt{G}(C, B, \Phi_{\Lambda_0})$ is the composition $S=V W$ on the space $L^2(G, \mathbb{C}^{n\times n})$ which is given by
\begin{align*}
S: \mathbf{f} \mapsto  \sum _{l\in \Lambda_0} \sum_{k\in \Lambda,m\in \Lambda'}\langle \mathbf{f},E_{Cm}T_{Bk}\Phi_l \rangle E_{Cm}T_{Bk}\Phi_l,
\end{align*}
 $ \mathbf{f} \in L^2(G, \mathbb{C}^{n\times n})$. The frame operator is  bounded, linear and invertible on $L^2(G, \mathbb{C}^{n\times n})$. We refer to \cite{OC2, Groch} for basic theory of Gabor frames.

 The following example will be used in illustration of results.

\begin{ex}\cite[Example 3.1]{Jind}\label{exb1}
Let $G$ be the torus group. Its  dual group is $\widehat{G}=\mathbb{Z}$. Fix a lattice $\Lambda=\Big\{0,\frac{1}{8},\dots, \frac{7}{8}\Big\}$. Then $\Lambda^{\perp}=8\mathbb{Z}$ with fundamental domain $V=\mathbb{Z}_8=\{0,1,\dots,7\}$. Let $\phi_1,\phi_2\in L^2(\mathbb{T})$ be such that
\begin{align*}
 \widehat{\phi_1}(\gamma)=\chi_{\mathbb{Z}_8}(\gamma) \quad  \text{and} \quad   \widehat{\phi_2}(\gamma)=\frac{1}{2}\chi_{\mathbb{Z}_8}(\gamma) \ \text{in} \ L^2(\mathbb{Z})\ \text{for} \ \gamma\in \mathbb{Z}.
 \end{align*}
For $B\in$ Aut$G$ and  $C\in$ Aut$\widehat{G}$, consider the Gabor system $\{E_{Cm}T_{Bk}\phi_1\}_{k\in \Lambda \atop  m \in \Lambda^{\perp}}$ $= \{E_{8m}T_{k}\phi_1\}_{k\in \Lambda, m \in \mathbb{Z}} $. Set $\phi_{m}^{(1)}(\xi)=E_{8m}\phi_1(\xi)$, $ m\in \mathbb{Z}$, $\xi\in [0,1[$. Since $E_{8m}T_k\phi_1(\xi)=T_kE_{8m}\phi_1(\xi)$, thus by taking
 $\Lambda_m := \Lambda$, one can write  \break  $\{T_k\phi_{m}^{(1)}\}_{k \in \Lambda,  m\in \mathbb{Z}}$ $= \{E_{8m}T_k\phi_1\}_{k\in \Lambda, m\in \mathbb{Z}}$.

 Define
 \begin{align*}
  G_0(\gamma) & =\sum_{m \in \mathbb{Z}} \mu_{\widehat{G}}(V)\Big|\widehat{\phi^{(1)}_{m}}(\gamma)\Big|^2, \ \gamma \in \mathbb{Z},
   \intertext{and}
 G_1(\gamma) & =\sum_{m \in \mathbb{Z}} \mu_{\widehat{G}}(V)\sum_{w\in \Lambda^{\perp}\backslash \{0\}}|\widehat{\phi^{(1)}_{m}}(\gamma)\widehat{\phi^{(1)}_{m}}(\gamma+w)|, \  \gamma \in \mathbb{Z}.
 \end{align*}

Then, using  $\widehat{\phi_m^{(1)}}(\gamma)=  \widehat{E_{8m}\phi_1}(\gamma)= T_{8m}\widehat{\phi_1}(\gamma)$, $\gamma \in \mathbb{Z}$,  we have
 \begin{align*}
&G_0(\gamma)=\sum_{m \in \mathbb{Z}}8|\chi_{\mathbb{Z}_8}(\gamma-8m)|^2=8 \ \text{for} \ \gamma \in \mathbb{Z},
 \intertext{and}
&G_1(\gamma)=\sum\limits_{m \in \mathbb{Z}}8\sum_{a\in \mathbb{Z} \backslash \{0\}}|\chi_{\mathbb{Z}_8}(\gamma-8m)\chi_{\mathbb{Z}_8}(\gamma+8a-8m)|^2=0 \ \text{for} \ \gamma \in \mathbb{Z}.
\end{align*}

Therefore,  by \cite[Theorem 21.6.1]{OC2}, the Gabor system  $\{E_{8m}T_k\phi_1\}_{k\in \Lambda, m \in \mathbb{Z}}$ is a $8$-tight frame for $L^2(G)$. Similarly,  $\{E_{8m}T_k\phi_2\}_{k\in \Lambda, m \in \mathbb{Z}}$ is a $2$-tight frame for $L^2(G)$.
\end{ex}
\section{Matrix-Valued $(\Theta, \Theta^*)$-Gabor Frames}\label{SecIII}
We begin this section with the definition of a  matrix-valued $(\Theta, \Theta^*)$-Gabor  frame in the  matrix-valued function space $L^2(G, \mathbb{C}^{n\times n})$.
\begin{defn} \label{def3.2x}
Let $\Theta$ be a bounded linear operator acting on  $L^2(G, \mathbb{C}^{n\times n})$. A countable family  of vectors   $\mathcal{G}(C, B, \Phi_{\Lambda_0}) : =  \{E_{Cm}T_{Bk}\Phi_l\}_{l\in \Lambda_0, k\in \Lambda,m\in \Lambda'}$ in $ L^2(G, \mathbb{C}^{n\times n})$ is called a \emph{matrix-valued $(\Theta, \Theta^*)$-Gabor frame}  for $L^2(G, \mathbb{C}^{n\times n})$ if for all  $\mathbf{f} \in L^2(G, \mathbb{C}^{n\times n})$,
\begin{align}\label{eq1}
\alpha_{o} \|\Theta^* \mathbf{f}\|^2\leq \sum _{l\in \Lambda_0}\sum_{k\in \Lambda,m\in \Lambda'}\Big\|\Big\langle  E_{Cm}T_{Bk}\Phi_l,\mathbf{f} \Big\rangle \Big\|^2 \leq \beta_{o} \|\Theta \mathbf{f}\|^2
\end{align}
holds for  some positive scalars $\alpha_{o}$ and $\beta_{o}$.
\end{defn}
The positive scalars $\alpha_{o}$ and  $\beta_{o}$ are called \textit{lower} and \emph{upper frame  bounds} of the $(\Theta, \Theta^*)$-Gabor frame $\mathcal{G}(C, B, \Phi_{\Lambda_0})$.
If $\alpha_{o} = \beta_{o}$, then we say that $\mathcal{G}(C, B, \Phi_{\Lambda_0})$  is a $\alpha_{o}$-$(\Theta, \Theta^*)$-tight matrix-valued Gabor frame for
$L^2(G, \mathbb{C}^{n\times n})$.

\begin{rem}
If $\Theta$ is the identity operator on $L^2(G, \mathbb{C}^{n\times n})$, then a  matrix-valued $(\Theta, \Theta^*)$-Gabor frame for $L^2(G, \mathbb{C}^{n\times n})$ is the standard  matrix-valued Gabor frame  for $L^2(G, \mathbb{C}^{n\times n})$. However, if $\Theta$ is a non-identity operator on $L^2(G, \mathbb{C}^{n\times n})$, then a  matrix-valued $(\Theta, \Theta^*)$-Gabor frame for $L^2(G, \mathbb{C}^{n\times n})$ need not be the standard  matrix-valued Gabor frame  for $L^2(G, \mathbb{C}^{n\times n})$. For example, consider the tight Gabor frames $\{E_{8m}T_{k}\phi_l\}_{k\in \Lambda,m\in \mathbb{Z}}, (l = 1,2)$ for $L^2(G)$ given in Example \ref{exb1}.
Let $\Phi_1=\begin{bmatrix}
0 & \phi_1 \\
0 & \phi_1
\end{bmatrix}, \Phi_2=\begin{bmatrix}
0 & \phi_2 \\
0 & \phi_2
\end{bmatrix}$. Then, $\Phi_1, \Phi_2 \in L^2(G,\mathbb{C}^{2\times 2})$.
For any  $\mathbf{f}=\begin{bmatrix}
f_{11}& f_{12}\\
f_{21}& f_{22}
\end{bmatrix}$
in $L^2(G,\mathbb{C}^{2\times 2})$,  we have
\begin{align*}
&\sum_{l\in\{1,2\}}\sum_{k\in \Lambda,  m\in\mathbb{Z}}\Big\|\Big\langle  E_{8m}T_{k}\Phi_l,\mathbf{f} \Big\rangle \Big\|^2\\
& = \sum_{l\in\{1,2\}} \sum_{k\in \Lambda,  m\in\mathbb{Z}}2\Big(\big|\int_{G}E_{8m}T_{k}\phi_l\overline{f_{12}} d\mu_{G}\big|^2+ \big|\int_{G}E_{8m}T_{k}\phi_l\overline{f_{22}} d\mu_{G}\big|^2 \Big)\\
&=20\Big(\|f_{12}\|^2+\|f_{22}\|^2\Big).
\end{align*}
Therefore, for  $\mathbf{f}_o=\begin{bmatrix}
f & 0\\
f & 0
\end{bmatrix}$, where  $0\neq f \in L^2(G)$, we have
\begin{align*}
\sum_{l\in\{1,2\}}\sum_{k\in \Lambda,  m\in\mathbb{Z}}\Big\|\Big\langle  E_{8m}T_{k}\Phi_l,\mathbf{f}_o \Big\rangle \Big\|^2=0.
\end{align*}
Thus, $\{E_{8m}T_{k}\Phi_l\}_{l\in \{1,2\},k\in \Lambda,m\in \mathbb{Z}}$  is  not a matrix-valued Gabor frame for $L^2(G,\mathbb{C}^{2\times 2})$. But the family  $\{E_{8m}T_{k}\Phi_l\}_{l\in \{1,2\},k\in \Lambda,m\in \mathbb{Z}}$ is  a $(\Theta_o, \Theta_o^*)$-Gabor frame for $L^2(G,\mathbb{C}^{2\times 2})$, where $\Theta_o$ is a bounded linear operator  on  $L^2(G,\mathbb{C}^{2\times 2})$  given by
\begin{align*}
\Theta_o: \mathbf{f} \mapsto \begin{bmatrix}
0 & f_{12}\\
0 & f_{22}
\end{bmatrix}, \ \mathbf{f} = \begin{bmatrix}
f_{11} & f_{12}\\
f_{21}& f_{22}
\end{bmatrix} \in L^2(G, \mathbb{C}^{2\times 2}).
\end{align*}
It is easy to see that  $\Theta_o^* = \Theta_o$.
Therefore, for any $\mathbf{f} \in L^2(G, \mathbb{C}^{2\times 2})$,  we have
\begin{align*}
20\|\Theta_o^* \mathbf{f}\|^2 = \sum_{l\in\{1,2\}} \sum_{k\in \Lambda, m\in\mathbb{Z}}\Big\|\Big\langle  E_{8m}T_{k}\Phi_l,\mathbf{f} \Big\rangle \Big\|^2= 20 \|\Theta_o\mathbf{f}\|^2.
\end{align*}
Hence, $\{E_{8m}T_{k}\Phi_l\}_{l\in \{1,2\},k\in \Lambda,m\in \mathbb{Z}}$
is  a  matrix-valued $(\Theta_o, \Theta_o^*)$-Gabor frame for $L^2(G,\mathbb{C}^{2\times 2})$.
\end{rem}
\begin{rem}\label{exper1}
It is mentioned in \cite{Jind} that a  matrix-valued Gabor frame for  $L^2(G,\mathbb{C}^{n\times n})$ is always a $\Theta$-Gabor frame  for $L^2(G,\mathbb{C}^{n\times n})$ where $\Theta$ is a  bounded linear operator on $L^2(G, \mathbb{C}^{n\times n})$. However, this is not true in the case of
$(\Theta, \Theta^*)$-matrix-valued Gabor frame. Precisely, a  matrix-valued Gabor frame for  $L^2(G,\mathbb{C}^{n\times n})$ need not be a $(\Theta, \Theta^*)$-Gabor frame  for $L^2(G,\mathbb{C}^{n\times n})$. For example, let $G$ be the torus group and $\{E_{8m}T_{k}\phi_l\}_{k\in \Lambda,m\in \mathbb{Z}}$ $(l = 1,2)$  be the tight Gabor frames for $L^2(G)$ given in Example \ref{exb1}. Let $\Phi_1$, $\Phi_2 \in L^2(G,\mathbb{C}^{2\times 2})$ be given by
\begin{align*}
\Phi_1=\begin{bmatrix}
0 &\phi_1 \\
\phi_2 & 0
\end{bmatrix} \quad \text{and} \quad  \Phi_2=\begin{bmatrix}
0 &\phi_2 \\
\phi_1 & 0
\end{bmatrix}.
\end{align*}
Then, $\{E_{8m}T_{k}\Phi_l\}_{l\in \{1,2\},k\in \Lambda,m\in \mathbb{Z}}$ is  a $10$-tight  matrix-valued Gabor  frame for $L^2(G,\mathbb{C}^{2\times 2})$.
Define
$\Theta: L^2(G,\mathbb{C}^{2\times 2}) \rightarrow L^2(G,\mathbb{C}^{2\times 2})$ by
\begin{align*}
\Theta \colon \mathbf{f} \mapsto \begin{bmatrix}
f_{11} & 0 \\
0 & 0
\end{bmatrix}, \ \mathbf{f} = \begin{bmatrix}
f_{11} & f_{12} \\
f_{21}& f_{22}
\end{bmatrix} \in L^2(G, \mathbb{C}^{2\times 2}).
\end{align*}
Then, $\Theta$ is a bounded linear operator. If possible, let $\{ E_{8m}T_{k}\Phi_l\}_{l\in \{1,2\},k\in \Lambda,m\in \mathbb{Z}}$ be a $(\Theta, \Theta^*)$- Gabor frame for $L^2(G,\mathbb{C}^{2\times 2})$ with bounds $a,b$. Then, for  $\mathbf{f}_o=\begin{bmatrix}
0 & f \\
f & f \\
\end{bmatrix}$, where  $0\neq f \in L^2(G)$, we have
\begin{align*}
\sum_{l\in\{1,2\}}\sum_{k\in \Lambda,  m\in\mathbb{Z}}\Big\|\Big\langle  E_{8m}T_{k}\Phi_l,\mathbf{f}_o \Big\rangle \Big\|^2=10 \|\mathbf{f}_o\|^2=30\|f\|^2 > 0= b\|\Theta \mathbf{f}_o \|^2,
\end{align*}
which is a contradiction. Hence, $\{ E_{8m}T_{k}\Phi_l\}_{l\in \{1,2\},k\in \Lambda,m\in \mathbb{Z}}$ is not a $(\Theta, \Theta^*)$- Gabor frame for $L^2(G,\mathbb{C}^{2\times 2})$.
\end{rem}
Now, we show that a matrix-valued Gabor frame for  $L^2(G,\mathbb{C}^{n\times n})$ becomes a $(\Theta, \Theta^*)$-Gabor frame  for $L^2(G,\mathbb{C}^{n\times n})$ provided  $\Theta$ is  bounded below.
\begin{prop}\label{prop3.4}
Let  $\{E_{Cm}T_{Bk}\Phi_l\}_{l\in \Lambda_0, k\in \Lambda, m\in \Lambda'}$ be a  matrix-valued Gabor frame for $L^2( G, \mathbb{C}^{n\times n})$. Let  $\Theta$ be a bounded linear operator acting on the space  $L^2( G, \mathbb{C}^{n\times n})$ which is bounded below. Then, the collection  $\{ E_{Cm}T_{Bk}\Phi_l\}_{l\in \Lambda_0, k\in \Lambda, m\in \Lambda'}$ is a matrix-valued $(\Theta, \Theta^*)$-Gabor frame for $L^2( G, \mathbb{C}^{n\times n})$.
\end{prop}
\begin{proof}
Let $\gamma$ and $\delta$ be frame bounds of $\{E_{Cm}T_{Bk}\Phi_l\}_{l\in \Lambda_0, k\in \Lambda, m\in \Lambda'}$. Let $\Theta$ be bounded below by a constant $\alpha$, that is, $\|\Theta \mathbf{f}\| \geq \alpha\|\mathbf{f}\|$ for all $\mathbf{f}$ in $L^2( G, \mathbb{C}^{n\times n})$.
Then, for any $ \mathbf{f} \in L^2( G, \mathbb{C}^{n\times n})$, we have
\begin{align*}
  \frac{\gamma}{\|\Theta^{*}\|^2}\|\Theta^{*} \mathbf{f}\|^2  \leq \gamma \| \mathbf{f}\|^2 \leq \sum _{l\in \Lambda_0}\sum_{k\in \Lambda,m\in \Lambda'}\|\langle   E_{Cm}T_{Bk}\Phi_l, \mathbf{f} \rangle\|^2,
\end{align*}
and
 \begin{align*}
 \sum _{l\in \Lambda_0}\sum_{k\in \Lambda,m\in \Lambda'}\|\langle E_{Cm}T_{Bk}\Phi_l, \mathbf{f}\rangle\|^2 \leq \delta \|\mathbf{f}\|^2
  \leq  \frac{\delta}{\alpha^2} \|\Theta \mathbf{f}\|^2.
\end{align*}
Thus, $\{E_{Cm}T_{Bk}\Phi_l\}_{l\in \Lambda_0, k\in \Lambda, m\in \Lambda'}$ is a matrix-valued $(\Theta, \Theta^*)$-Gabor frame for the space
$L^2( G, \mathbb{C}^{n\times n})$ with frame bounds $ \frac{\gamma}{\|\Theta^{*}\|^2}$ and $\frac{\delta}{\alpha^2}$.
\end{proof}

Now, we discuss  relations between  hyponormal operators  on $L^2(G, \mathbb{C}^{n\times n})$ and matrix-valued $\lambda_{o}$-$(\Theta, \Theta^*)$-tight  frames for $L^2(G, \mathbb{C}^{n\times n})$. By Definition \ref{def3.2x}, one may observe that a  bounded linear operator $\Theta$ on $L^2(G, \mathbb{C}^{n\times n})$ is hyponormal if there exists  a matrix-valued $\lambda_{o}$-$(\Theta, \Theta^*)$-tight  frame for the space  $L^2(G, \mathbb{C}^{n\times n})$. Indeed, if  $\{\mathbf{f}_{ k}\}_{k \in I}$ is a matrix-valued $\lambda_{o}$-$(\Theta, \Theta^*)$-tight  frame for $L^2(G, \mathbb{C}^{n\times n})$,
then by Definition \ref{def3.2x}, we have  $\|\Theta^*\mathbf{f}\| \leq \|\Theta \mathbf{f}\|$, for all $\mathbf{f} \in L^2(G, \mathbb{C}^{n\times n})$. Hence, $\Theta$  is a hyponormal operator on  $L^2(G, \mathbb{C}^{n\times n})$.

In order to see the other way round relationship, we first discuss some frame preserving properties of $(\Theta, \Theta^*)$-frames in $L^2(G, \mathbb{C}^{n\times n})$.  The following result says that the image of a frame in $L^2(G)$ under a hyponormal operator $\Theta$ is a $(\Theta, \Theta^*)$-frame for $L^2(G)$.
\begin{prop}\label{prop1}
Let $\{E_{Cm}T_{Bk}\phi_l\}_{l\in \Lambda_0, k\in \Lambda, m\in \Lambda'}$ be a Gabor frame for $L^2(G)$ and let $\Theta$ be a hyponormal operator on $L^2(G)$. Then,
$\{\Theta E_{Cm}T_{Bk}\phi_l\}_{l\in \Lambda_0, k\in \Lambda \atop  m\in \Lambda'}$ is a $(\Theta, \Theta^*)$-frame for $L^2(G)$.
\end{prop}
\begin{proof}
Let $\lambda$ and $\mu$ be lower and upper frame bounds of $\{E_{Cm}T_{Bk}\phi_l\}_{l\in \Lambda_0, k\in \Lambda \atop  m\in \Lambda'}$. Then, using the hyponormality of $\Theta$, for any $f \in L^2(G)$, we have
\begin{align*}
\sum _{l\in \Lambda_0}\sum_{k\in \Lambda,m\in \Lambda'}|\langle \Theta E_{Cm}T_{Bk}\phi_l, f \rangle|^2  & =\sum _{l\in \Lambda_0}\sum_{k\in \Lambda, m\in \Lambda'}|\langle E_{Cm}T_{Bk}\phi_l, \Theta^{*} f \rangle|^2\\
 &\leq \mu \|\Theta^{*} f\|^2\\
&\leq \mu \|\Theta f\|^2.
\intertext{Also}
 \lambda \|\Theta^{*} f\|^2 &\leq \sum _{l\in \Lambda_0}\sum_{k\in \Lambda,m\in \Lambda'}|\langle E_{Cm}T_{Bk}\phi_l, \Theta^{*} f \rangle|^2 \\
 &=\sum _{l\in \Lambda_0}\sum_{k\in \Lambda,m\in \Lambda'}|\langle \Theta E_{Cm}T_{Bk}\phi_l, f \rangle|^2
\end{align*}
 for all $f \in L^2(G)$. Hence, $\{\Theta E_{Cm}T_{Bk}\phi_l\}_{l\in \Lambda_0, k\in \Lambda, m\in \Lambda'}$ is a $(\Theta, \Theta^*)$-frame  for $L^2(G)$ with frame bounds $\lambda$ and $\mu$.
\end{proof}

\begin{rem}
Proposition \ref{prop1} is not true for matrix-valued frames in matrix-valued signal spaces $L^2(G, \mathbb{C}^{n\times n})$. This problem is related to adjointable operators on matrix-valued signal spaces with respect to matrix-valued inner product on the underlying space. For example, consider the tight Gabor frames  $\{E_{8m}T_{k}\phi_l\}_{k\in \Lambda \atop m \in \mathbb{Z}}$  $(l = 1, 2)$ for $L^2(G)$ given in Example \ref{exb1}. Let $\Phi_1$, $\Phi_2 \in L^2(G,\mathbb{C}^{2\times 2})$ be given by
\begin{align*}
  \Phi_1=\begin{bmatrix}
0 &\phi_1 \\
\phi_2 & 0
\end{bmatrix} \quad \text{and} \quad  \Phi_2=\begin{bmatrix}
  0 &\phi_2 \\
  \phi_1 & 0
\end{bmatrix}.
\end{align*}
Then, $\{E_{8m}T_{k}\Phi_l\}_{l\in \{1,2\},k\in \Lambda,m\in \mathbb{Z}}$ is  a $10$-tight  matrix-valued Gabor  frame for $L^2(G,\mathbb{C}^{2\times 2})$.

 Define
$\Theta: L^2(G,\mathbb{C}^{2\times 2}) \rightarrow L^2(G,\mathbb{C}^{2\times 2})$ by
\begin{align*}
\Theta \colon \mathbf{f} \mapsto \begin{bmatrix}
  f_{11} & 0 \\
  0 & 0
  \end{bmatrix}, \ \mathbf{f} = \begin{bmatrix}
f_{11} & f_{12} \\
f_{21}& f_{22}
 \end{bmatrix} \in L^2(G, \mathbb{C}^{2\times 2}).
 \end{align*}
Then, $\Theta$ is a bounded linear operator with adjoint $\Theta^*=\Theta$ and hence a  hyponormal operator. But $\Theta$ is not adjointable with respect to matrix-valued inner product on $L^2(G, \mathbb{C}^{2\times 2})$. That is, $\langle \Theta \mathbf{f},\mathbf{g}\rangle \ne \langle \mathbf{f}, \Theta^*\mathbf{g}\rangle$ for all $\mathbf{f},\mathbf{g}$ in $ L^2(G, \mathbb{C}^{2\times 2})$. Furthermore,  $\Theta E_{8m}T_{k}\Phi_l= \mathbf{O}$ for  $l\in \{1,2\}, k\in \Lambda, m\in \mathbb{Z}$. Hence, $\{\Theta E_{8m}T_{k}\Phi_l\}_{l\in \{1,2\},k\in \Lambda,m\in \mathbb{Z}}$ is not a $(\Theta, \Theta^*)$-frame for $L^2(G,\mathbb{C}^{2\times 2})$.
\end{rem}
The following result gives sufficient conditions on matrix-valued $\Theta$-frame preserving transformations acting on matrix-valued signal spaces in terms of adjointability of  $\Theta$.
 \begin{prop}\label{1a}
Let  $\{E_{Cm}T_{Bk}\Phi_l\}_{l\in \Lambda_0, k\in \Lambda, m\in \Lambda'}$ be a matrix-valued frame for the space $L^2( G, \mathbb{C}^{n\times n})$ with frame bounds $\gamma$ and $\delta$. Let  $\Theta$ be a hyponormal operator acting on $L^2( G, \mathbb{C}^{n\times n})$ which is adjointable with respect to the matrix-valued inner product. Then,  $\{\Theta E_{Cm}T_{Bk}\Phi_l\}_{l\in \Lambda_0, k\in \Lambda, m\in \Lambda'}$ is a matrix-valued
$(\Theta, \Theta^*)$-frame for $L^2( G, \mathbb{C}^{n\times n})$ with frame bounds $\gamma$ and $\delta$.
\end{prop}
\begin{proof}
For any $ \mathbf{f} \in L^2( G, \mathbb{C}^{n\times n})$, we have
\begin{align*}
 \gamma \|\Theta^{*} \mathbf{f}\|^2  &\leq \sum _{l\in \Lambda_0}\sum_{k\in \Lambda,m\in \Lambda'}\|\langle  E_{Cm}T_{Bk}\Phi_l, \Theta^{*} \mathbf{f}          \rangle\|^2\\
& = \sum _{l\in \Lambda_0}\sum_{k\in \Lambda,m\in \Lambda'}\|\langle  \Theta E_{Cm}T_{Bk}\Phi_l, \mathbf{f}\rangle\|^2\\
 &= \sum _{l\in \Lambda_0}\sum_{k\in \Lambda,m\in \Lambda'}\|\langle E_{Cm}T_{Bk}\Phi_l, \Theta^{*} \mathbf{f} \rangle\|^2\\
 & \leq \delta \|\Theta^{*} \mathbf{f}\|^2\\
 & \leq  \delta \|\Theta \mathbf{f}\|^2.
\end{align*}
Thus, $\{\Theta E_{Cm}T_{Bk}\Phi_l\}_{l\in \Lambda_0, k\in \Lambda, m\in \Lambda'}$ is a matrix-valued $(\Theta, \Theta^*)$-frame for the space $L^2( G, \mathbb{C}^{n\times n})$ with the desired frame bounds.
\end{proof}
Now, we have enough knowledge to discuss the conditions on an operator $\Theta$ acting on $L^2( G, \mathbb{C}^{n\times n})$ such that the existence of $\lambda_{o}$-$(\Theta, \Theta^*)$-tight frames for $L^2(G, \mathbb{C}^{n\times n})$ is guaranteed. We give the following result regarding this.
\begin{thm}\label{thm2}
Let  $\Theta$ be a hyponormal operator on $L^2(G, \mathbb{C}^{n\times n})$. If $\Theta$ is  adjointable with respect to the matrix-valued inner product, then there exists  a matrix-valued $\lambda_{o}$-$(\Theta, \Theta^*)$-tight frame for $L^2(G, \mathbb{C}^{n\times n})$ for every positive real number $\lambda_{o}$.
\end{thm}
\begin{proof}
Let  $\{E_{Cm}T_{Bk} \phi_l\}_{l\in \Lambda_0, k\in \Lambda, m\in \Lambda'}$ be a Parseval frame for $L^2(G)$. For each $l\in \Lambda_0$, define the matrix-valued function  $ \Phi_l \in L^2( G, \mathbb{C}^{n\times n})$ as
\begin{align*}
\Phi_l=\begin{bmatrix}
\sqrt{\lambda_{o}} \ \phi_l & 0 & \cdots & 0\\
0 & \sqrt{\lambda_{o}} \ \phi_l & \cdots & 0\\
\vdots & \vdots & \ddots & \vdots\\
0 & 0 & \cdots & \sqrt{\lambda_{o}} \ \phi_l
\end{bmatrix}.
\end{align*}
Then
\begin{align*}
 E_{Cm}T_{Bk}\Phi_l = \tiny{\begin{bmatrix}
E_{Cm}T_{Bk}(\sqrt{\lambda_{o}} \ \phi_l) & 0 & \cdots & 0\\
0 & E_{Cm}T_{Bk}(\sqrt{\lambda_{o}} \ \phi_l) & \cdots & 0\\
\vdots & \vdots & \ddots & \vdots\\
0 & 0 & \cdots & E_{Cm}T_{Bk}(\sqrt{\lambda_{o}} \ \phi_l)
\end{bmatrix}}.
\end{align*}
Therefore, for any $ \mathbf{f} = \begin{bmatrix}
f_{11} & f_{12} & \cdots & f_{1n}\\
f_{21} & f_{22} & \cdots & f_{2n}\\
\vdots & \vdots & \ddots & \vdots\\
f_{n1} & f_{n2} & \cdots & f_{nn}
\end{bmatrix} \in L^2( G, \mathbb{C}^{n\times n})$, we have

\begin{align*}
&\sum _{l\in \Lambda_0}\sum_{k\in \Lambda,m\in \Lambda'}\|\langle  \mathbf{f},  E_{Cm}T_{Bk}\Phi_l \rangle\|^2\\
& =\sum _{l\in \Lambda_0}\sum_{k\in \Lambda \atop m\in \Lambda'} \Big\| \begin{bmatrix}
\langle f_{11},E_{Cm}T_{Bk}(\sqrt{\lambda_{o}} \ \phi_l)\rangle &  \cdots & \langle f_{1n}, E_{Cm}T_{Bk}(\sqrt{\lambda_{o}} \ \phi_l)\rangle\\
\langle f_{21}, E_{Cm}T_{Bk}(\sqrt{\lambda_{o}} \ \phi_l) \rangle  &  \cdots & \langle f_{2n}, E_{Cm}T_{Bk}(\sqrt{\lambda_{o}} \ \phi_l) \rangle \\
\vdots & \ddots & \vdots \\
\langle f_{n1}, E_{Cm}T_{Bk}(\sqrt{\lambda_{o}} \ \phi_l) \rangle  &  \cdots & \langle f_{nn}, E_{Cm}T_{Bk}(\sqrt{\lambda_{o}} \ \phi_l) \rangle
\end{bmatrix} \Big\|^2\\
&= \lambda_{o}\sum _{l\in \Lambda_0}\sum_{k\in \Lambda, m\in \Lambda'} \sum_{1\leq i,j\leq n}| \langle f_{ij},E_{Cm}T_{Bk}\ \phi_l\rangle|^2\\
&= \lambda_{o}\sum_{1\leq i,j\leq n}\| f_{ij}\|^2\\
&=\lambda_{o} \|\mathbf{f}\|^2.
\end{align*}
Hence, $\{E_{Cm}T_{Bk}\Phi_l\}_{l\in \Lambda_0,k\in \Lambda,m\in \Lambda'}$ is a matrix-valued $\lambda_{o}$-tight  Gabor frame for $L^2( G, \mathbb{C}^{n\times n})$. Further, by Proposition \ref{1a}, $\{\Theta E_{Cm}T_{Bk}\Phi_l\}_{l\in \Lambda_0,  k\in \Lambda,  m \in \Lambda'}$ is a $(\Theta, \Theta^*)$-frame for $L^2( G, \mathbb{C}^{n\times n})$ with $\lambda_{o}$ as lower and upper frame bounds.
Hence, the existence of a matrix-valued  $\lambda_{o}$-$(\Theta, \Theta^*)$-tight frame for the space $L^2(G, \mathbb{C}^{n\times n})$ is proved.
\end{proof}
We illustrate Theorem \ref{thm2} by giving the following example regarding the existence of $\lambda_{o}$-$(\Theta, \Theta^*)$-tight frames for $L^2(\mathbb{R}, \mathbb{C}^{3\times 3})$.
\begin{ex}
Let $G=\mathbb{R}$ be the additive group of real numbers. The characters on $\mathbb{R}$ are the functions $\eta_y:\mathbb{R}\rightarrow \mathbb{C}$ defined by
\begin{align*}
\eta_y(x)=\textit{e}^{2\pi \textit{i} y x},\ x \in \mathbb{R}
\end{align*}
 for fixed $y \in \mathbb{R}$. That is, the  dual group  $\widehat{G}$ can be identified with $\mathbb{R}$, see \cite{Foll} for technical details. Consider the lattice $\Lambda=\mathbb{Z}$ and $\Lambda'=\mathbb{Z}$.
Then, for $\phi= \chi_{[0,1]}$, the Gabor system $\{E_m T_k\phi\}_{m, k \in \mathbb{Z}}$ is an orthonormal basis for $L^2(\mathbb{R})$, see  \cite[p. 96]{OC2} for details.

Define a matrix-valued function  $ \Phi \in L^2( \mathbb{R}, \mathbb{C}^{3\times 3})$ as
\begin{align*}
\Phi=\begin{bmatrix}
\sqrt{3} \ \phi & 0 & 0\\
0 & \sqrt{3} \ \phi  & 0\\
0 & 0 & \sqrt{3} \ \phi
\end{bmatrix}.
\end{align*}
Then, for any $ \mathbf{f} = \big[f_{i,j}\big]_{1 \leq i, j \leq n} \in L^2( \mathbb{R}, \mathbb{C}^{3\times 3})$, we have
\begin{align*}
\sum_{m, k \in \mathbb{Z}}\|\langle  \mathbf{f}, E_m T_k\Phi \rangle\|^2 =3 \|\mathbf{f}\|^2,
\end{align*}
which implies that $\{E_m T_k\Phi\}_{m, k \in \mathbb{Z}}$ is a matrix-valued $3$-tight Gabor frame for $L^2( \mathbb{R}, \mathbb{C}^{3\times 3})$.
Define $\Theta: L^2(\mathbb{R},\mathbb{C}^{3\times 3}) \rightarrow L^2(\mathbb{R},\mathbb{C}^{3\times 3})$ by
\begin{align*}
\Theta \colon \mathbf{f} \mapsto \begin{bmatrix}
f_{11} & 0& f_{13} \\
f_{21}& 0& f_{23}\\
f_{31}& 0& f_{33}
\end{bmatrix}, \ \mathbf{f} = \begin{bmatrix}
f_{11} & f_{12}& f_{13} \\
f_{21}& f_{22}& f_{23}\\
f_{31}& f_{32}& f_{33}
 \end{bmatrix} \in L^2(\mathbb{R}, \mathbb{C}^{3\times 3}).
 \end{align*}
Then, $\Theta$ is a bounded linear operator with adjoint $\Theta^*=\Theta$. Also, $\Theta$ is adjointable with respect to matrix-valued inner product on $L^2(\mathbb{R}, \mathbb{C}^{3\times 3})$. That is, $\langle \Theta \mathbf{f},\mathbf{g}\rangle = \langle \mathbf{f}, \Theta^*\mathbf{g}\rangle, \ \mathbf{f},\mathbf{g} \in L^2(\mathbb{R}, \mathbb{C}^{3\times 3})$.
Then, by Proposition \ref{1a}, the matrix-valued system $\{\Theta E_m T_k\Phi\}_{m, k \in \mathbb{Z}}$ is a matrix-valued 3-$(\Theta, \Theta^*)$-tight frame for $L^2( \mathbb{R}, \mathbb{C}^{3\times 3})$.
\end{ex}

\begin{rem}\label{rm7}
In Theorem \ref{thm2}, the condition of adjointability of $\Theta$ with respect to matrix-valued inner product is not a necessary condition.
\end{rem}

Next, we discuss frame properties of the image of a $(\Theta, \Theta^*)$-Gabor frame in $L^2(G, \mathbb{C}^{n\times n})$ under a bounded linear operator $\Xi$. It is proved in  \cite[Proposition 4.2]{Jind} that the image of a $\Theta$-Gabor frame for $L^2( G, \mathbb{C}^{n\times n})$ under an operator $\Xi \in \mathcal{B}(L^2( G, \mathbb{C}^{n\times n}))$ becomes a $\Xi \Theta$-frame for $L^2( G, \mathbb{C}^{n\times n})$ provided $\Xi$ is  adjointable with respect to matrix-valued inner product. But, this is not true for the case of $(\Theta, \Theta^*)$-Gabor frames in $L^2( G, \mathbb{C}^{n\times n})$. That is, if $\mathcal{G}(C, B, \Phi_{\Lambda_0})$ is a  $(\Theta, \Theta^*)$-Gabor frame for $L^2( G, \mathbb{C}^{n\times n})$ and $\Xi \in \mathcal{B}(L^2( G, \mathbb{C}^{n\times n}))$ is  adjointable with respect to matrix-valued inner product, then $\Xi(\mathcal{G}(C, B, \Phi_{\Lambda_0}))$ may not be a $(\Xi \Theta, (\Xi \Theta)^*)$-frame for $L^2( G, \mathbb{C}^{n\times n})$. This is justified in the following example.
\begin{ex}\label{ex2}
 Consider tight Gabor frames $\{E_{8m}T_k\phi_1\}_{k\in \Lambda \atop  m \in \mathbb{Z}}$ and  $\{E_{8m}T_k\phi_2\}_{k\in \Lambda \atop  m \in \mathbb{Z}}$  for $L^2(G)$ given in Example \ref{exb1}. Define
$\Theta: L^2(G,\mathbb{C}^{2\times 2}) \rightarrow L^2(G,\mathbb{C}^{2\times 2})$ by
\begin{align*}
\Theta \colon \mathbf{f} \mapsto \begin{bmatrix}
  f_{22} &  f_{21} \\
   f_{12} &  f_{11}
  \end{bmatrix}, \ \mathbf{f} = \begin{bmatrix}
f_{11} & f_{12} \\
f_{21}& f_{22}
 \end{bmatrix} \in L^2(G, \mathbb{C}^{2\times 2}).
 \end{align*}
Then, $\{E_{8m}T_{k}\Phi_l\}_{l\in \{1,2\},k\in \Lambda,m\in \mathbb{Z}}$ is a matrix-valued $10$-tight $(\Theta, \Theta^*)$-frame for $L^2(G, \mathbb{C}^{2\times 2})$. In fact, for any $\mathbf{f} \in L^2(G, \mathbb{C}^{2\times 2})$, we have
\begin{align*}
10\|\Theta^*\mathbf{f}\|^2&=10\|\mathbf{f}\|^2
\leq  \sum_{l\in\{1,2\}} \sum\limits_{k\in \Lambda,m\in\mathbb{Z}} \|\langle \mathbf{f},E_{8m}T_{k}\Phi_l\rangle\|^2&\leq 10 \|\mathbf{f}\|^2
= 10 \|\Theta\mathbf{f}\|^2.
\end{align*}

 Define
$\Xi: L^2(G,\mathbb{C}^{2\times 2}) \rightarrow L^2(G,\mathbb{C}^{2\times 2})$ by
\begin{align*}
\Xi \colon \mathbf{f} \mapsto \begin{bmatrix}
  0&  f_{12} \\
  0 &  f_{22}
  \end{bmatrix}, \ \mathbf{f} = \begin{bmatrix}
f_{11} & f_{12} \\
f_{21}& f_{22}
 \end{bmatrix} \in L^2(G, \mathbb{C}^{2\times 2}).
 \end{align*}
Then, $\Xi$ is a bounded linear operator with adjoint $\Xi^*=\Xi$. Also, $\Xi$ is  adjointable with respect to matrix-valued inner product on $L^2(G, \mathbb{C}^{2\times 2})$. That is, $\langle \Xi \mathbf{f},\mathbf{g}\rangle = \langle \mathbf{f}, \Xi^*\mathbf{g}\rangle, \ \mathbf{f},\mathbf{g} \in L^2(G, \mathbb{C}^{2\times 2})$. However, $\{\Xi E_{8m}T_{k}\Phi_l\}_{l\in \{1,2\},k\in \Lambda \atop m\in \mathbb{Z}}$ is not a  $(\Xi \Theta, (\Xi\Theta)^*)$-frame. If possible, let $\{\Xi E_{8m}T_{k}\Phi_l\}_{l\in \{1,2\},k\in \Lambda,m\in \mathbb{Z}}$ be a  $(\Xi \Theta, (\Xi\Theta)^*)$-frame with bounds $\gamma, \delta$. Then, for  $\mathbf{f}_o=\begin{bmatrix}
0 & f \\
0 & f \\
\end{bmatrix}$, where  $f$ is a non-zero function in $L^2(G)$, we have
\begin{align*}
\sum_{l\in\{1,2\}}\sum_{k\in \Lambda,  m\in\mathbb{Z}}\Big\|\Big\langle \Xi E_{8m}T_{k}\Phi_l,\mathbf{f}_o \Big\rangle \Big\|^2= 20\|f\|^2 >0 = \delta \|\Xi\Theta \mathbf{f}_o \|^2,
\end{align*}
which is a contradiction.
\end{ex}
In the following result, we give some additional conditions on $\Xi$ so that $\Xi(\mathcal{G}(C, B, \Phi_{\Lambda_0}))$ becomes a $(\Xi \Theta, (\Xi \Theta)^*)$-frame. This result can be seen as a generalization of Proposition \ref{1a}.
\begin{prop}\label{2a}
Let $\{E_{Cm}T_{Bk}\Phi_l\}_{l\in \Lambda_0, k\in \Lambda, m\in \Lambda'}$ be  a matrix-valued  $(\Theta, \Theta^*)$-Gabor frame for $L^2( G, \mathbb{C}^{n\times n})$ with frame bounds $\gamma$ and $\delta$. Suppose
\begin{enumerate}[$(i)$]
  \item $\Xi \in \mathcal{B}(L^2( G, \mathbb{C}^{n\times n}))$ is  adjointable with respect to matrix-valued inner product.
  \item $\Xi$ is  hyponormal on Ran$(\Theta)$ such that $\Theta \Xi^*= \Xi^* \Theta$.
\end{enumerate}
 Then, $\{\Xi E_{Cm}T_{Bk}\Phi_l\}_{l\in \Lambda_0, k\in \Lambda, m\in \Lambda'}$is a $(\Xi \Theta, (\Xi \Theta)^*)$-frame for $L^2( G, \mathbb{C}^{n\times n})$ with the same frame bounds.
\end{prop}
\proof
For any $\mathbf{f}\in L^2(G,\mathbb{C}^{n\times n})$, we have
\begin{align}\label{e3}
\sum_{l\in \Lambda_0} \sum_{k\in \Lambda,m\in \Lambda'}\Big\|\Big\langle \Xi E_{Cm}T_{Bk}\Phi_l,\mathbf{f} \Big\rangle \Big\|^2 &= \sum_{l\in \Lambda_0} \sum_{k\in \Lambda, m\in \Lambda'}\Big\|\Big\langle  E_{Cm}T_{Bk}\Phi_l, \Xi^*\mathbf{f} \Big\rangle \Big\|^2\notag\\
 & \leq  \delta\|\Theta \Xi^*\mathbf{f}\|^2\notag\\
& =  \delta\| \Xi^* \Theta\mathbf{f}\|^2\notag\\
& \leq \delta\| \Xi \Theta\mathbf{f}\|^2.
\end{align}
Similarly
\begin{align}\label{e4}
\sum_{l\in \Lambda_0} \sum_{k\in \Lambda,m\in \Lambda'}\Big\|\Big\langle \Xi E_{Cm}T_{Bk}\Phi_l,\mathbf{f} \Big\rangle \Big\|^2 &= \sum_{l\in \Lambda_0} \sum_{k\in \Lambda, m\in \Lambda'}\Big\|\Big\langle  E_{Cm}T_{Bk}\Phi_l, \Xi^*\mathbf{f} \Big\rangle \Big\|^2\notag\\
 & \geq \gamma \|\Theta^* \Xi^*\mathbf{f}\|^2\notag\\
 &= \gamma \|(\Xi \Theta)^*\mathbf{f}\|^2 \ \mathbf{f} \in L^2(G,\mathbb{C}^{n\times n}).
\end{align}
By  $(\ref{e3})$ and $(\ref{e4})$, we conclude that $\{\Xi E_{Cm}T_{Bk}\Phi_l\}_{l\in \Lambda_0, k\in \Lambda, m\in \Lambda'}$ is a matrix-valued $(\Xi \Theta, (\Xi \Theta)^*)$-frame for $L^2( G, \mathbb{C}^{n\times n})$ with frame bounds $\gamma$ and $\delta$. This completes the proof.
\endproof
\begin{rem}
The condition that the operator $\Theta$ commutes  with  $\Xi^*$ in Theorem \ref{2a} cannot be relaxed. Consider the operators $\Theta$, $\Xi$ defined on $L^2(G, \mathbb{C}^{2\times 2})$ and the system  $\{E_{8m}T_{k}\Phi_l\}_{l\in \{1,2\},k\in \Lambda,m\in \mathbb{Z}}$ which is a $(\Theta, \Theta^*)$-Gabor frame  for $L^2(G, \mathbb{C}^{2\times 2})$ given in Example \ref{ex2}. As mentioned in Example \ref{ex2}, the operator $\Xi$ is adjointable with respect to matrix-valued inner product, and $\Xi$ is  hyponormal on Ran$(\Theta)$ since $\Xi^*=\Xi$. But, $\Theta \Xi^* \ne \Xi^* \Theta$. In fact, for any $\mathbf{f} = \begin{bmatrix}
f_{11} & f_{12} \\
f_{21}& f_{22}
 \end{bmatrix} \in L^2(G, \mathbb{C}^{2\times 2})$, we have
\begin{align*}
\Theta \Xi^* \mathbf{f}=\Theta \Xi \mathbf{f} = \begin{bmatrix}
   f_{22}& 0  \\
  f_{12}& 0
  \end{bmatrix} \quad
\text{and} \quad
\Xi^* \Theta \mathbf{f}=\Xi \Theta  \mathbf{f} = \begin{bmatrix}
0&   f_{21}  \\
 0& f_{11}
  \end{bmatrix}.
\end{align*}
Therefore, the system $\{\Xi E_{8m}T_{k}\Phi_l\}_{l\in \{1,2\},k\in \Lambda,m\in \mathbb{Z}}$  not being a  $(\Xi \Theta, (\Xi\Theta)^*)$-frame, details in Example \ref{ex2}, supports our argument.
\end{rem}

Next, we give a  characterization for   matrix-valued  $(\Theta, \Theta^*)$-Gabor frames in $L^2(G, \mathbb{C}^{n\times n})$. This is inspired by a fundamental result due to G\v{a}vruta in \cite[Theorem 4]{LG} for ordinary $K$-frames in separable Hilbert spaces. This is also related with the concept of atomic systems in Hilbert spaces. The matrix-valued atomic system in matrix-valued function spaces can be studied in terms of  $(\Theta, \Theta^*)$-Gabor frames.  In the following result, if $\mathcal{G}(C, B, \Phi_{\Lambda_0})$ is a Parseval $(\Theta, \Theta^*)$-Gabor frame of $L^2( G, \mathbb{C}^{n\times n})$, then $\Theta$ turns out to be a hyponormal operator on $L^2( G, \mathbb{C}^{n\times n})$.

\begin{thm}\label{3.5}
Let  $\Theta$ be a bounded linear operator acting on $L^2(G, \mathbb{C}^{n\times n})$. A matrix-valued Gabor system $\mathcal{G}(C, B, \Phi_{\Lambda_0})$ is a $(\Theta, \Theta^*)$-Gabor frame for the space $L^2( G, \mathbb{C}^{n\times n})$ if and only if there exists a bounded linear  operator $\Omega$ from  $\ell^2(\Lambda_0\times\Lambda\times\Lambda', \mathcal{M}_{n}(\mathbb{C}))$ into $ L^2(G, \mathbb{C}^{n\times n})$ such that
\begin{enumerate}[$(i)$]
\item $ E_{Cm} T_{Bk}\Phi_l= \Omega \chi_{l,k,m}, \ l\in\Lambda_0,k\in \Lambda,m\in\Lambda'$, where $\{ \chi_{l,k,m} \}_{l\in\Lambda_0,k\in \Lambda \atop  m \in\Lambda'} $ is an orthonormal basis of $\ell^2(\Lambda_0\times\Lambda \times \Lambda',\mathcal{M}_n(\mathbb{C}))$,\label{conI}
\item  there exist finite positive numbers $\alpha$ and $\beta$ satisfying
\begin{align*}
\alpha \ \text{tr} \langle \Theta \Theta^*\mathbf{f},\mathbf{f} \rangle \leq  \text{tr}\langle  \Omega \Omega^*\mathbf{f},\mathbf{f} \rangle \leq
\beta \ \text{tr} \langle \Theta^*\Theta \mathbf{f},\mathbf{f} \rangle, \ \mathbf{f} \in L^2(G, \mathbb{C}^{n\times n}).
\end{align*} \label{con2}
\end{enumerate}
\end{thm}
\proof
Suppose first  that $\mathcal{G}(C, B, \Phi_{\Lambda_0})$ is a matrix-valued $(\Theta, \Theta^*)$-Gabor frame for the space  $L^2(G, \mathbb{C}^{n\times n})$ with frame bounds  $a_o$, $b_o$.

Define  $\Xi:L^2(G, \mathbb{C}^{n\times n}) \rightarrow \ell^2( \Lambda_0\times\Lambda\times\Lambda',\mathcal{M}_n(\mathbb{C}))$  by
\begin{align*}
\Xi (\mathbf{f}) =  \sum_{l\in\Lambda_0} \sum_{k\in\Lambda, m\in\Lambda'}\langle  \mathbf{f}, E_{Cm} T_{Bk}\Phi_l \rangle \chi_{l,k,m}, \ \mathbf{f} \in L^2(G, \mathbb{C}^{n\times n}).
\end{align*}
Then, $\Xi$ is a bounded linear operator and $\|\Xi\| \leq \sqrt{b_o}\|\Theta\|$.

 Now, for any  $l\in\Lambda_0,k\in\Lambda,m\in\Lambda'$, we have
\begin{align*}
\text{tr} \langle  \chi_{l,k,m}, \Xi \mathbf{f} \rangle &=  \text{tr} \Big\langle  \chi_{l,k,m}, \sum_{l'\in\Lambda_0}\sum_{k'\in\Lambda, m'\in\Lambda'}\langle  \mathbf{f},  E_{Cm'} T_{Bk'}\Phi_{l'}\rangle \chi_{l',k',m'}\Big \rangle\\
&= \text{tr} \langle  \mathbf{f},  E_{Cm} T_{Bk}\Phi_l \rangle^*\\
&= \text{tr} \langle  E_{Cm} T_{Bk}\Phi_l, \mathbf{f} \rangle \ \text{for all} \ \mathbf{f} \in L^2(G, \mathbb{C}^{n\times n}).
\end{align*}
Thus,  $\Xi^* \chi_{l,k,m}= E_{Cm} T_{Bk}\Phi_l$, for all $l\in\Lambda_0$, $k\in\Lambda$ and $m\in\Lambda'$. If we take $\Omega= \Xi^*$, then we obtain  \eqref{conI}. To prove  \eqref{con2}, let
$\mathbf{f} \in L^2(G, \mathbb{C}^{n\times n})$ be arbitrary. Then,
\begin{align*}
a_o \|\Theta^*\mathbf{f}\|^2 \leq  \sum_{l\in\Lambda_0}\sum_{k\in\Lambda, m\in\Lambda'}\Big\|\langle  \mathbf{f},  E_{Cm} T_{Bk}\Phi_l \rangle \Big\|^2=\|\Xi \mathbf{f}\|^2,\\
\intertext{and}
\|\Xi \mathbf{f}\|^2=\sum_{l\in\Lambda_0}\sum_{k\in\Lambda, m\in\Lambda'}\Big\|\langle  \mathbf{f},  E_{Cm} T_{Bk}\Phi_l \rangle \Big\|^2 \leq b_o \|\Theta\mathbf{f}\|^2.
\end{align*}
This imply that
\begin{align*}
a_o \text{tr} \langle \Theta \Theta^*\mathbf{f},\mathbf{f}\rangle \leq \text{tr}  \langle \Xi^* \Xi \mathbf{f},\mathbf{f}\rangle =\text{tr} \langle \Omega\Omega^*\mathbf{f}, \mathbf{f} \rangle \leq b_o \text{tr} \langle \Theta^* \Theta\mathbf{f},\mathbf{f}\rangle.
 \end{align*}
 This gives \eqref{con2}, where  $\alpha = a_o$ and  $\beta=b_o$.

To prove the converse, assume that conditions  \eqref{conI} and \eqref{con2}  hold. Then, using condition \eqref{conI}, for any $ l\in\Lambda_0,k\in\Lambda,m\in\Lambda'$ and any $\mathbf{f} \in L^2(G, \mathbb{C}^{n\times n})$, we have
\begin{align*}
\text{tr} \langle  \chi_{l,k,m}, \Omega^*\mathbf{f} \rangle &=  \text{tr} \langle \Omega \chi_{l,k,m}, \mathbf{f} \rangle\\
&= \text{tr} \langle  E_{C_m} T_{Bk}\Phi_l ,\mathbf{f} \rangle\\
&= \text{tr} \langle  \mathbf{f},  E_{Cm} T_{Bk}\Phi_l \rangle^*\\
 &=\text{tr} \Big\langle  \chi_{l,k,m},  \sum_{l'\in\Lambda_0}\sum_{k'\in\Lambda, m'\in\Lambda'}\langle  \mathbf{f},  E_{Cm'} T_{Bk'}\Phi_{l'}\rangle \chi_{l',k',m'}\Big \rangle,
\end{align*}
which entails
\begin{align*}
\Omega^*{\mathbf{f}}=  \sum_{l\in\Lambda_0} \sum_{k\in\Lambda,m\in\Lambda'}\langle  \mathbf{f},  E_{Cm} T_{Bk}\Phi_l\rangle \chi_{l,k,m}, \ \mathbf{f} \in L^2(G, \mathbb{C}^{n\times n}).
\end{align*}
Therefore
\begin{align*}
\|\Omega^*\mathbf{f}\|^2 =  \sum_{l\in\Lambda_0}\sum_{k\in\Lambda,m\in\Lambda'}\Big\|\langle  \mathbf{f},  E_{Cm} T_{Bk}\Phi_l\rangle \Big\|^2,  \  \mathbf{f} \in L^2(G, \mathbb{C}^{n\times n}).
\end{align*}
Using condition \eqref{con2},  for all $\mathbf{f} \in L^2(G, \mathbb{C}^{n\times n})$, we have
\begin{align*}
\alpha \|\Theta^*\mathbf{f}\|^2 \leq \|\Omega^*\mathbf{f}\|^2=  \sum_{l\in\Lambda_0} \sum_{k\in\Lambda,m\in\Lambda'}\Big\|\langle  \mathbf{f},  E_{Cm} T_{Bk}\Phi_l \rangle \Big\|^2 \leq \beta \|\Theta\mathbf{f}\|^2
\end{align*}
Hence, $\mathcal{G}(C, B, \Phi_{\Lambda_0})$ is a matrix-valued $(\Theta, \Theta^*)$-Gabor frame for $L^2(G, \mathbb{C}^{n\times n})$. This completes the proof.
\endproof
\section{Perturbation of $(\Theta, \Theta^*)$-Gabor Frames}\label{SecIV}
In this section, we show that matrix-valued $(\Theta, \Theta^*)$-Gabor frames are stable under small perturbation. Perturbation theory plays a significant role in both pure mathematics and applied science, see e.g. \cite{Kato76}. For applications of perturbation theory for frames in various directions, we refer to \cite{H11}.  The following  result shows that  multivariate $(\Theta, \Theta^*)$-Gabor frames in  matrix-valued signal spaces are stable under small perturbations.
\begin{thm}\label{pert}
Let $\mathcal{G}(C, B, \Phi_{\Lambda_0})$ be  a matrix-valued  $(\Theta, \Theta^*)$-Gabor frame for $L^2( G, \mathbb{C}^{n\times n})$ with frame bounds $\gamma_o$, $\delta_o$, and let  $\{\widetilde{\Phi_l}\}_{l\in \Lambda_0} \subset L^2(G,\mathbb{C}^{n\times n})$. Assume that
\begin{enumerate}[$(i)$]
 \item $\Theta^*$ be bounded below by  $m_o$.\label{percon1}
 \item  $\lambda$, $\mu$, $\eta \geq 0$  be  such that $\frac{(1-2\lambda)\gamma_o-2\mu}{2\eta}> \frac{\|\Theta\|^2}{m_o^2}$.
 \item For all  $\mathbf{f} \in L^2(G,\mathbb{C}^{n\times n})$,
\begin{align}\label{hypo11}
\sum_{l\in \Lambda_0}\sum\limits_{k\in \Lambda, m\in \Lambda'} \|\langle \mathbf{f},E_{Cm}T_{Bk}(\Phi_l- \widetilde{\Phi_l})\rangle\|^2
 &\leq \lambda \sum_{l\in \Lambda_0}\sum\limits_{k\in \Lambda, m\in \Lambda'} \|\langle \mathbf{f}, E_{Cm}T_{Bk}\Phi_l \rangle\|^2 \notag\\
 & + \mu \|\Theta^*\mathbf{f} \|^2 + \eta \|\Theta \mathbf{f} \|^2.
\end{align}
\end{enumerate}
Then, $\mathcal{G}(C, B, \widetilde{\Phi}_{\Lambda_0})$ is a  matrix-valued  $(\Theta, \Theta^*)$-Gabor frame for $L^2(G,\mathbb{C}^{n\times n})$ with frame bounds
\begin{align*}
\left(\Big(\frac{1}{2}-\lambda\Big)\gamma_o- \mu-\frac{\eta\|\Theta\|^2}{m_o^2}\right)  \ \text{and} \  2\left(\Big(1+\lambda+\frac{\mu}{\gamma_o}\Big)\delta_o+\eta\right).
\end{align*}

\end{thm}
\proof
By hypothesis \eqref{hypo11}, for any  $\mathbf{f} \in L^2(G,\mathbb{C}^{n\times n})$, we have
\begin{align*}
&\sum_{l\in \Lambda_0} \sum\limits_{k\in \Lambda, m\in \Lambda'} \|\langle \mathbf{f}, E_{Cm}T_{Bk}\widetilde{\Phi_l} \rangle\|^2\\
&\leq  2 \sum_{l\in \Lambda_0} \sum\limits_{k\in \Lambda, m\in \Lambda'} \|\langle \mathbf{f}, E_{Cm}T_{Bk}\Phi_l- E_{Cm}T_{Bk}\widetilde{\Phi_l}\rangle\|^2 \\
& + 2\sum_{l\in \Lambda_0}\sum\limits_{k\in \Lambda, m\in \Lambda'} \|\langle \mathbf{f}, E_{Cm}T_{Bk}\Phi_l \rangle\|^2\\
&\leq (2\lambda+2) \sum_{l\in \Lambda_0} \sum\limits_{k\in \Lambda, m\in \Lambda'} \|\langle \mathbf{f}, E_{Cm}T_{Bk}\Phi_l\rangle\|^2 \\
& + 2 \mu\|\Theta^*\mathbf{f}\|^2 +2\eta\|\Theta\mathbf{f}\|^2 \\
&\leq (2\lambda+2) \sum_{l\in \Lambda_0}\sum\limits_{k\in \Lambda, m\in \Lambda'} \|\langle \mathbf{f}, E_{Cm}T_{Bk}\Phi_l\rangle\|^2 \\
& + \frac{2\mu}{\gamma_o} \sum_{l\in \Lambda_0} \sum\limits_{k\in \Lambda, m\in \Lambda'} \|\langle \mathbf{f},  E_{Cm}T_{Bk} \Phi_l\rangle\|^2 + 2\eta\|\Theta\mathbf{f}\|^2.
\end{align*}
Therefore
\begin{align}\label{per1}
\sum_{l\in \Lambda_0}\sum\limits_{k\in \Lambda,m\in \Lambda'} \|\langle \mathbf{f}, E_{Cm}T_{Bk}\widetilde{\Phi_l} \rangle\|^2
&\leq 2\left(\Big(1+\lambda+\frac{\mu}{\gamma_o}\Big)\delta_o+\eta\right) \|\Theta \mathbf{f}\|^2,
\end{align}
for all $\mathbf{f} \in L^2(G,\mathbb{C}^{n\times n})$.

Similarly,
\begin{align*}
&\sum_{l\in \Lambda_0} \sum\limits_{k\in \Lambda, m\in \Lambda'} \|\langle \mathbf{f}, E_{Cm}T_{Bk}\Phi_l \rangle\|^2\\
&\leq  2 \sum_{l\in \Lambda_0} \sum\limits_{k\in \Lambda, m\in \Lambda'} \|\langle \mathbf{f}, E_{Cm}T_{Bk}\Phi_l- E_{Cm}T_{Bk}\widetilde{\Phi_l}\rangle\|^2 \\
&+ 2\sum_{l\in \Lambda_0}\sum\limits_{k\in \Lambda, m\in \Lambda'} \|\langle \mathbf{f}, E_{Cm}T_{Bk}\widetilde{\Phi_l} \rangle\|^2\\
&\leq 2\lambda \sum_{l\in \Lambda_0} \sum\limits_{k\in \Lambda, m\in \Lambda'} \|\langle \mathbf{f}, E_{Cm}T_{Bk}\Phi_l\rangle\|^2 +
+2\mu\|\Theta^*\mathbf{f}\|^2 +2\eta\|\Theta\mathbf{f}\|^2 \\
& + 2\sum_{l\in \Lambda_0}\sum\limits_{k\in \Lambda, m\in \Lambda'} \|\langle \mathbf{f}, E_{Cm}T_{Bk}\widetilde{\Phi_l} \rangle\|^2,
\end{align*}
 which entails
\begin{align}\label{per2}
 & 2\sum_{l\in \Lambda_0}\sum\limits_{k\in \Lambda, m\in \Lambda'} \|\langle \mathbf{f}, E_{Cm}T_{Bk}\widetilde{\Phi_l} \rangle\|^2 \notag\\
 &\geq (1-2\lambda) \sum_{l\in \Lambda_0} \sum\limits_{k\in \Lambda, m\in \Lambda'} \|\langle \mathbf{f}, E_{Cm}T_{Bk}\Phi_l\rangle\|^2 -2\mu\|\Theta^*\mathbf{f}\|^2 -2\eta\|\Theta\mathbf{f}\|^2\notag\\
 &\geq(1-2\lambda)\gamma_o\|\Theta^*\mathbf{f}\|^2-2\mu\|\Theta^*\mathbf{f}\|^2 -2\eta\|\Theta\|^2\|\mathbf{f}\|^2\notag\\
&\geq(1-2\lambda)\gamma_o\|\Theta^*\mathbf{f}\|^2-2\mu\|\Theta^*\mathbf{f}\|^2 -\frac{2\eta\|\Theta\|^2}{m_o^2}\|\Theta^*\mathbf{f}\|^2 \quad \big(\text{using hypothesis} \  \eqref{percon1}\big)\notag\\
\intertext{That is}
& \sum_{l\in \Lambda_0}\sum\limits_{k\in \Lambda, m\in \Lambda'} \|\langle \mathbf{f}, E_{Cm}T_{Bk}\widetilde{\Phi_l} \rangle\|^2 \geq
\left(\Big(\frac{1}{2}-\lambda\Big)\gamma_o- \mu-\frac{\eta\|\Theta\|^2}{m_o^2}\right)\|\Theta^*\mathbf{f}\|^2
\end{align}
for all $ \mathbf{f} \in L^2(G,\mathbb{C}^{n\times n})$. From \eqref{per1} and \eqref{per2}, we conclude that
 $\mathcal{G}(C, B, \widetilde{\Phi}_{\Lambda_0})$ is a frame for $L^2(G,\mathbb{C}^{n\times n})$ with the desired frame bounds.
\endproof
Next is an applicative example of Theorem \ref{pert}.
\begin{ex}\label{pertexa}
Let $\{E_{8m}T_{k}\Phi_l\}_{l\in \{1,2\},k\in \Lambda,m\in \mathbb{Z}}$ be the $10$-tight matrix-valued Gabor  frame for $L^2(G,\mathbb{C}^{2\times 2})$ given in Remark \ref{exper1}. Define $\Theta$ on  $L^2(G, \mathbb{C}^{2\times 2})$ by
\begin{align*}
\Theta: \mathbf{f}
\mapsto \begin{bmatrix}
2f_{22} & f_{21}\\
f_{12} & f_{11}
\end{bmatrix}, \ \mathbf{f}= \begin{bmatrix}
f_{11} & f_{12} \\
f_{21} & f_{22}
\end{bmatrix} \in L^2(G, \mathbb{C}^{2\times 2}).
\end{align*}
Then, $\Theta$ is a bounded linear operator satisfying $\|\Theta \mathbf{f}\|\leq 2 \|\mathbf{f}\|$, for all $\mathbf{f} \in L^2(G, \mathbb{C}^{2\times 2})$. In fact, we have $\|\Theta\|=2$. For any $\mathbf{f}$, $\mathbf{g}\in L^2(G,\mathbb{C}^{2\times 2})$, we have
\begin{align*}
\text{tr}\langle \Theta\mathbf{f},\mathbf{g}\rangle  &=\text{tr} \int_{G} \begin{bmatrix}
2f_{22} & f_{21}\\
f_{12} & f_{11}
\end{bmatrix}\begin{bmatrix}
\overline{g_{11}} & \overline{g_{21}}\\
\overline{g_{12}} & \overline{g_{22}}
\end{bmatrix}d\mu_{G}\\
&= \text{tr} \int_{G} \begin{bmatrix}
f_{11} & f_{12} \\
f_{21}  & f_{22}
\end{bmatrix}\begin{bmatrix}
\overline{g_{22}}&\overline{g_{12}}\\
\overline{g_{21}} & 2\ \overline{g_{11}}
\end{bmatrix}d\mu_{G},
\end{align*}
which implies that $\Theta^*$ is given by
\begin{align*}
\Theta^*: \mathbf{g}
\mapsto \begin{bmatrix}
g_{22} & g_{21}\\
g_{12} & 2g_{11}
\end{bmatrix}, \ \mathbf{g}= \begin{bmatrix}
g_{11} & g_{12} \\
g_{21} & g_{22}
\end{bmatrix} \in L^2(G, \mathbb{C}^{2\times 2}).
\end{align*}
It can be easily seen that $\Theta^*$ satisfies $ \|\mathbf{f}\|\leq \|\Theta^* \mathbf{f}\|\leq 2 \|\mathbf{f}\|$, for all $\mathbf{f} \in L^2(G, \mathbb{C}^{2\times 2})$. That is, $\Theta^*$ is bounded below by $m_o=1$.

 Now,   for any $\mathbf{f}\in L^2(G,\mathbb{C}^{2\times 2})$, we have
\begin{align*}
\frac{5}{2}\|\Theta^* \mathbf{f}\|^2 \leq \sum_{l\in\{1,2\}}\sum_{k\in \Lambda, m\in\mathbb{Z}}\Big\|\Big\langle E_{8m}T_{k}\Phi_l,\mathbf{f} \Big\rangle \Big\|^2 = 10 \|\mathbf{f}\|^2 \leq 10 \|\Theta\mathbf{f}\|^2.
\end{align*}
Therefore, $\{E_{8m}T_{k}\Phi_l\}_{l\in \{1,2\},k\in \Lambda,m\in \mathbb{Z}}$ is  a matrix-valued  $(\Theta, \Theta^*)$-Gabor frame for $L^2(G,\mathbb{C}^{2\times 2})$ with frame bounds $\gamma_1=\frac{5}{2}$ and $\delta_1=10$.

Consider $\widetilde{\Phi_1}=\begin{bmatrix}
\frac{1}{5}\phi_1 & \phi_1 \\
\phi_2 & \frac{1}{5}\phi_2
\end{bmatrix}$ ,  $\widetilde{\Phi_2}=\begin{bmatrix}
\frac{1}{5}\phi_2& \phi_2 \\
\phi_1 & \frac{1}{5}\phi_1
\end{bmatrix}$ in $L^2(G,\mathbb{C}^{2\times 2})$.
Then, for any $\mathbf{f} = \Big[f_{i j}\Big]_{1 \leq i, j \leq 2} \in L^2(G,\mathbb{C}^{2\times 2})$, we have
\begin{align*}
& \sum_{l\in\{1,2\}} \sum\limits_{k\in \Lambda,m\in\mathbb{Z}} \|\langle \mathbf{f},E_{8m}T_{k}\Phi_l - E_{8m}T_{k}\widetilde{\Phi_l}\rangle\|^2\\
&= \frac{1}{25} \sum_{l\in\{1,2\}}\sum_{k\in \Lambda,m\in\mathbb{Z}}\Big(\big|\int_{G}E_{8m}T_{k}\phi_l\overline{f_{11}}\big|^2+ \big|\int_{G}E_{8m}T_{k}\phi_l\overline{f_{21}}\big|^2 \\
& +   \big|\int_{G}E_{8m}T_{k}\phi_l\overline{f_{12}}\big|^2+\big|\int_{G}E_{8m}T_{k}\phi_l\overline{f_{22}}\big|^2\Big)\\
& = \frac{10}{25} \|\textbf{f}\|^2\\
& \leq \frac{1}{5} \|\Theta^*\mathbf{f} \|^2 + \frac{1}{5} \|\Theta \mathbf{f} \|^2.
\end{align*}
Thus,  all  the conditions in Theorem \ref{pert} are satisfied with $\lambda=0,\mu=\frac{1}{5}, \eta=\frac{1}{5}$. Hence, the collection $\{E_{8m}T_{k} \widetilde{\Phi_l}\}_{l\in \{1,2\},k\in \Lambda,m\in \mathbb{Z}}$ is a matrix-valued  $(\Theta, \Theta^*)$-Gabor frame for $L^2(G,\mathbb{C}^{2\times 2})$.
\end{ex}
Theorem \ref{pert} shows that a matrix-valued  Gabor system $\mathcal{G}(C, B, \widetilde{\Phi}_{\Lambda_0})$ becomes a  $(\Theta, \Theta^*)$-Gabor frame for $L^2(G,\mathbb{C}^{n\times n})$ if its window functions $\widetilde{\Phi}_l, l \in \Lambda_0$ are sufficiently close to the window functions $\Phi_l, l \in \Lambda_0$ of a matrix-valued $(\Theta, \Theta^*)$-Gabor frame $\mathcal{G}(C, B, \Phi_{\Lambda_0})$. This can also be seen as a way of constructing new matrix-valued $(\Theta, \Theta^*)$-Gabor frames by altering the window functions of a known matrix-valued $(\Theta, \Theta^*)$-Gabor frame  appropriately. In the direction of obtaining new matrix-valued $(\Theta, \Theta^*)$-Gabor frames from known matrix-valued $(\Theta, \Theta^*)$-Gabor frames, we give the following result which states that the perturbed  matrix-valued  Gabor systems $(\Theta, \Theta^*)$-Gabor frames,  under suitable conditions,  becomes a matrix-valued $(\Theta, \Theta^*)$-Gabor frame for $L^2(G,\mathbb{C}^{n\times n})$.
\begin{thm}\label{sum}
Let $\mathcal{G}(C, B, \Phi_{\Lambda_0})$  and $\mathcal{G}(C, B, \Psi_{\Lambda_0})$ be matrix-valued  $(\Theta, \Theta^*)$-Gabor frames for $L^2( G, \mathbb{C}^{n\times n})$ with frame bounds $\gamma_1$, $\delta_1$ and $\gamma_2$, $\delta_2$, respectively.
Suppose $\Theta^*$ is  bounded below with constant $m_o$ such that $\sqrt{\frac{\gamma_1}{\delta_2}}> \frac{\|\Theta\|}{m_o}$.
Then, the perturbed matrix-valued Gabor system $\mathcal{G}\Big(C, B, (\Phi_{\Lambda_0} + \Psi_{\Lambda_0})\Big)$ is a $(\Theta, \Theta^*)$-Gabor frame for $L^2(G,\mathbb{C}^{n\times n})$ with frame bounds
\begin{align*}
\left(\sqrt{\gamma_1}-\frac{\sqrt{\delta_2}\|\Theta\|}{m_o}\right)^2 \  \text{and} \ \  2(\delta_1+ \delta_2).
\end{align*}
\end{thm}
\proof
For any  $\mathbf{f}\in L^2(G,\mathbb{C}^{n\times n})$, we compute
\begin{align*}
&\Big(\sum\limits_{l\in \Lambda_0}\sum_{k \in \Lambda, m \in \Lambda^ \prime}\|\langle \textbf{f }, E_{Cm}T_{Bk} (\Phi_l+\Psi_l) \rangle \|^2\Big)^\frac{1}{2}\\
&= \Big(\sum\limits_{l\in \Lambda_0}\sum_{k \in \Lambda, m \in \Lambda^ \prime}\|\langle \textbf{f }, E_{Cm}T_{Bk}\Phi_l+ E_{Cm}T_{Bk}\Psi_l \rangle \|^2\Big)^\frac{1}{2}\\
&\geq \Big(\sum\limits_{l\in \Lambda_0}\sum_{k \in \Lambda, m \in \Lambda^ \prime}\|\langle \textbf{f }, E_{Cm}T_{Bk}\Phi_l\rangle\|^2\Big)^\frac{1}{2} \\
&-\Big(\sum\limits_{l\in \Lambda_0}\sum_{k \in \Lambda, m \in \Lambda^ \prime}\|\langle \textbf{f }, E_{Cm}T_{Bk}\Psi_l \rangle \|^2\Big)^\frac{1}{2}\notag\\
&\geq \sqrt{\gamma_1} \|\Theta^*\mathbf{f}\|-\sqrt{\delta_2}\|\Theta\mathbf{f}\| \\
&\geq \sqrt{\gamma_1} \|\Theta^*\mathbf{f}\|-\sqrt{\delta_2}\|\Theta\|\|\mathbf{f}\|\\
&\geq\sqrt{\gamma_1} \|\Theta^*\mathbf{f}\|-\frac{\sqrt{\delta_2}\|\Theta\|}{m_o}\|\Theta^*\mathbf{f}\|.
\end{align*}
This  gives
\begin{align}\label{eq4xa}
&\sum\limits_{l\in \Lambda_0}\sum_{k \in \Lambda, m \in \Lambda^ \prime}\|\langle \textbf{f }, E_{Cm}T_{Bk}\Phi_l+ E_{Cm}T_{Bk}\Psi_l \rangle \|^2 \notag\\
& \geq \Big(\sqrt{\gamma_1}-\frac{\sqrt{\delta_2}\|\Theta\|}{m_o}\Big)^2\|\Theta^*\mathbf{f}\|^2 \ \text{for all} \ \mathbf{f}\in L^2(G,\mathbb{C}^{n\times n}).
\end{align}
Similarly,
\begin{align}\label{eq4xb}
&\sum\limits_{l\in \Lambda_0}\sum_{k \in \Lambda, m \in \Lambda^ \prime}\|\langle \textbf{f }, E_{Cm}T_{Bk} (\Phi_l+ \Psi_l) \rangle \|^2 \notag\\
&\leq 2\Big(\sum\limits_{l\in \Lambda_0}\sum_{k \in \Lambda, m \in \Lambda^ \prime}\|\langle \textbf{f }, E_{Cm}T_{Bk}\Phi_l\rangle\|^2  + \sum\limits_{l\in \Lambda_0}\sum_{k \in \Lambda, m \in \Lambda^ \prime}\|\langle \textbf{f }, E_{Cm}T_{Bk}\Psi_l \rangle \|^2\Big) \notag \\
&\leq 2(\delta_1+ \delta_2)\|\Theta\mathbf{f}\|^2 \ \text{for all} \ \mathbf{f}\in L^2(G,\mathbb{C}^{n\times n}).
\end{align}
From (\ref{eq4xa}) and (\ref{eq4xb}), we conclude that  $\mathcal{G}\Big(C, B, (\Phi_{\Lambda_0} + \Psi_{\Lambda_0})\Big)$  is a $(\Theta, \Theta^*)$-Gabor frame for $L^2(G,\mathbb{C}^{n\times n})$ with the desired frame bounds. This completes the proof.
\endproof
We end this paper by providing an application of Theorem \ref{sum}.
\begin{ex}\label{sumexa}
Consider the $(\Theta, \Theta^*)$-Gabor frame $\{E_{8m}T_{k}\Phi_l\}_{l\in \{1,2\},k\in \Lambda,m\in \mathbb{Z}}$ for $L^2(G,\mathbb{C}^{2\times 2})$ with frame bounds $\gamma_1=\frac{5}{2}$ and $\delta_1=10$ given in Example \ref{pertexa}.

Let $\Psi_1=\begin{bmatrix}
\frac{1}{5}\phi_1 & 0\\
0 & \frac{1}{5}\phi_2
\end{bmatrix}$ ,  $\Psi_2=\begin{bmatrix}
\frac{1}{5}\phi_2& 0\\
0 & \frac{1}{5}\phi_1
\end{bmatrix}$. Then, $\Psi_1$, $\Psi_2 \in L^2(G,\mathbb{C}^{2\times 2})$, and for any $\mathbf{f}\in L^2(G,\mathbb{C}^{2\times 2})$, we have
\begin{align*}
& \sum_{l\in\{1,2\}} \sum\limits_{k\in \Lambda,m\in\mathbb{Z}} \|\langle \mathbf{f},E_{8m}T_{k}\Psi_l\rangle\|^2\\
&= \frac{1}{25} \sum_{l\in\{1,2\}}\sum_{k\in \Lambda,m\in\mathbb{Z}}\Big(\big|\int_{G}E_{8m}T_{k}\phi_l\overline{f_{11}}\big|^2+ \big|\int_{G}E_{8m}T_{k}\phi_l\overline{f_{21}}\big|^2 \\
&+   \big|\int_{G}E_{8m}T_{k}\phi_l\overline{f_{12}}\big|^2+\big|\int_{G}E_{8m}T_{k}\phi_l\overline{f_{22}}\big|^2\Big)\\
& = \frac{10}{25} \|\textbf{f}\|^2\\
& \leq \frac{2}{5} \|\Theta \textbf{f}\|^2.
\end{align*}
Also
\begin{align*}
 \sum_{l\in\{1,2\}} \sum\limits_{k\in \Lambda,m\in\mathbb{Z}} \|\langle \mathbf{f},E_{8m}T_{k}\Psi_l\rangle\|^2
= \frac{2}{5} \| \textbf{f}\|^2
 \geq \frac{1}{10} \|\Theta^* \textbf{f}\|^2,  \  \  \mathbf{f}\in L^2(G,\mathbb{C}^{2\times 2}).
\end{align*}
Thus, $\{E_{8m}T_{k}\Psi_l\}_{l\in \{1,2\},k\in \Lambda,m\in \mathbb{Z}}$ is a $(\Theta, \Theta^*)$-Gabor frame for $L^2(G,\mathbb{C}^{2\times 2})$ with frame bounds $\gamma_2=\frac{1}{10}$ and $\delta_2=\frac{2}{5}$. Further,  $\Theta^*$ is bounded below by $m_o=1$ and $\frac{5}{2}=\sqrt{\frac{\gamma_1}{\delta_2}}> \frac{\|\Theta\|}{m_o}=2$. Hence, by Theorem \ref{sum}, the perturbed matrix-valued Gabor system
$\{E_{8m}T_{k} (\Phi_l+ \Psi_l)\}_{l\in \{1,2\},k\in \Lambda,m\in \mathbb{Z}}$ is a $(\Theta, \Theta^*)$-Gabor frame for $L^2(G,\mathbb{C}^{2\times 2})$ with frame bounds
\begin{align*}
\left(\sqrt{\gamma_1}-\frac{\sqrt{\delta_2}\|\Theta\|}{m_o}\right)^2   \quad  \text{and} \quad 2(\delta_1+ \delta_2)  .
\end{align*}
\end{ex}
\begin{rem}
Theorem \ref{pert} and Theorem \ref{sum} are not only ways of  constructing new frames but also can be used to check if a matrix-valued  Gabor system $\mathcal{G}(C, B, \widetilde{\Phi}_{\Lambda_0})$ is a $(\Theta, \Theta^*)$-Gabor frame for $L^2( G, \mathbb{C}^{n\times n})$, where the window functions $\widetilde{\Phi_l}, l \in {\Lambda_0}$ have complex structure leading to complicated calculations. In order to understand this better, we compare Example \ref{pertexa} and Example \ref{sumexa}. In Example \ref{pertexa}, to prove matrix-valued $(\Theta, \Theta^*)$-Gabor frame conditions of $\{E_{8m}T_{k}\widetilde{\Phi_l}\}_{l\in \{1,2\}, k\in \Lambda, m\in \mathbb{Z}}$, a $(\Theta, \Theta^*)$-Gabor frame $\{E_{8m}T_{k}\Phi_l\}_{l\in \{1,2\}, k\in \Lambda,m\in \mathbb{Z}}$ having simpler window functions is considered. However, in Example \ref{sumexa},  $\Phi_l+\Psi_l=\widetilde{\Phi_l}, l \in \Lambda_0$. Hence,  Example \ref{sumexa} can be seen as a method by which  the collection $\{E_{8m}T_{k}\widetilde{\Phi_l}\}_{l\in \{1,2\}, k\in \Lambda \atop  m\in \mathbb{Z}}$ is proved to be a $(\Theta, \Theta^*)$-Gabor frame by splitting its window functions as a sum of the window functions (perturbed window functions) of two $(\Theta, \Theta^*)$-Gabor frames.
\end{rem}
$$\textbf{\text{Data Related Statement}}$$
No data is used in this study.

$$\textbf{\text{Conflicts of Interest}}$$
The authors have no conflicts of interest.

\end{document}